\documentclass[11pt]{amsart}
\usepackage[letterpaper,top=3cm,bottom=2cm,left=3cm,right=3cm,marginparwidth=1.75cm]{geometry}

\usepackage[english]{babel}
\usepackage{amssymb, amsthm, amsfonts, amsmath}
\usepackage{cite}
\usepackage{empheq}
\usepackage{graphicx}
\usepackage{xcolor}
\usepackage{pagecolor}
\usepackage{tikz}

\definecolor{lime}{HTML}{A6CE39}
\DeclareRobustCommand{\orcidicon}{%
	\begin{tikzpicture}
		\draw[lime, fill=lime] (0,0) 
		circle [radius=0.16] 
		node[white] {{\fontfamily{qag}\selectfont \tiny ID}};
		\draw[white, fill=white] (-0.0625,0.095) 
		circle [radius=0.007];
	\end{tikzpicture}
	\hspace{-2mm}
}

\definecolor{myorange}{HTML}{FFA500}
\definecolor{myblue}{HTML}{5B9BD5}

\usepackage[colorlinks=true, allcolors=blue]{hyperref}

\usepackage{mathtools}
\mathtoolsset{showonlyrefs}

\numberwithin{equation}{section}

\newtheorem{theorem}{Theorem}
\newtheorem{corollary}[theorem]{Corollary}

\newtheorem{lemma}[theorem]{Lemma}
\newtheorem{proposition}[theorem]{Proposition}

\theoremstyle{remark}
\newtheorem{remark}[theorem]{Remark}

\numberwithin{equation}{section}
\numberwithin{theorem}{section}
\numberwithin{figure}{section}

\newcommand{\intT}{\int_0^1}

\newcommand{\IE}[1]{\intT e^{#1}} % {\rm d} x}
\newcommand{\I}[1]{\intT #1} %{\rm d} x}
\newcommand{\SI}[1]{\sum_{#1=0}^\infty}
\newcommand{\SIO}[1]{\sum_{#1=1}^\infty}

\newcommand{\T}{\mathbb{T}}
\newcommand{\Lo}{_{L^1}}
\newcommand{\Li}{_{L^\infty}}
\newcommand{\Lt}{_{L^2}}
\renewcommand{\Re}[1]{{\rm Re }(#1)}

\newcommand{\dy}{dy}

\newcommand{\N}{\mathbb{N}}
\newcommand{\R}{\mathbb{R}}

\newenvironment{AlignedEquation}{\begin{equation}\begin{aligned}}{\end{aligned}\end{equation}\ignorespacesafterend}

\newcommand{\qtext}[1]{\quad \text{#1} \quad}

\newcommand{\EL}{\mu}
\newcommand{\ELS}{\lambda}

\newcommand{\ENS}{\nu}
\newcommand{\FL}{f}
\newcommand{\FLS}{\psi}
\newcommand{\FNS}{\varphi}

\newcommand{\FR}[1]{R_{#1} \EL_{#1}}
\newcommand{\FRD}[1]{\overline R_{#1} \FL_{#1}}

\begin{document}

\title[Stability and bifurcation analysis in a mechanochemical model]{Stability and bifurcation analysis in a mechanochemical model of pattern formation}

\author[S. Cygan]{Szymon Cygan$^{1,3}$  \href{https://orcid.org/0000-0002-8601-829X}{\orcidicon}}
% \address[S. Cygan]{Institute for Mathematics, Heidelberg University, Germany\\
% Instytut Matematyczny, Uniwersytet Wroc\l{}awski, Poland \\ 
% \href{https://orcid.org/0000-0002-8601-829X}{orcid.org/0000-0002-8601-829X}
% }
% \email{szymon.cygan@uni-heidelberg.de}
% \urladdr {http://www.math.uni.wroc.pl/~scygan}

\author[A. Marciniak-Czochra]{Anna Marciniak-Czochra$^{1,2,*}$ \href{https://orcid.org/0000-0002-5831-6505}{\orcidicon}}
% \address[A. Marciniak-Czochra]{Institute for Mathematics and IWR, Heidelberg University, Germany\\ 
% \href{https://orcid.org/0000-0002-5831-6505}{orcid.org/0000-0002-5831-6505}
% }
% \email{anna.marciniak@iwr.uni-heidelberg.de}
% \urladdr{\href{https://biostruct.iwr.uni-heidelberg.de/folder_people/Anna.Marciniak/index.html}{https://biostruct.iwr.uni-heidelberg.de/folder\_people/Anna.Marciniak/index.html}}

\author[F. M\"unnich]{Finn M\"unnich$^1$ \href{https://orcid.org/0009-0007-9384-2002}{\orcidicon}}
% \address[F. M\"unnich]{Institute for Mathematics, Heidelberg University, Germany\\
% \href{https://orcid.org/0009-0007-9384-2002}{orcid.org/0009-0007-9384-2002}
% }
% \email{finn.muennich@stud.uni-heidelberg.de}

\author[D. Oelz]{Dietmar Oelz$^4$ \href{https://orcid.org/0000-0001-6981-4350}{\orcidicon}}

\thanks{$^1$ Institute for Mathematics, Heidelberg University, Germany \\
$^2$ Interdisciplinary Center for Scientific Computing (IWR), Heidelberg University, Germany\\
$^3$ Instytut Matematyczny, Uniwersytet Wroc\l{}awski, Poland \\
$^4$ School of Mathematics and Physics, University of Queensland, Australia\\
$^*$ Corresponding author; Email: anna.marciniak@iwr.uni-heidelberg.de\\
This work was supported by the German Research Foundation (DFG) under Germany’s Excellence Strategy (EXC 2181/1 – 390900948, Heidelberg STRUCTURES Excellence Cluster) and through the Collaborative Research Center 1324 (SFB 1324, Project B6). Additional funding was provided by the European Research Council (ERC) under the Synergy Grant (PEPS, No. 10107178), and by the German Academic Exchange Service (DAAD) within the Project-Based Personnel Exchange Programme (PPP, Project No. 57820507).
}

\begin{abstract}
We analyze the stability and bifurcation structure of steady states in a mechanochemical model of pattern formation in regenerating tissue spheroids. The model couples morphogen dynamics with tissue mechanics via a positive feedback loop: mechanical stretching enhances morphogen production, while morphogen concentration modulates tissue elasticity. Global strain conservation implements a nonlocal inhibitory effect, realizing a mechanochemical variant of the local activation–long-range inhibition mechanism. For exponential elasticity–morphogen coupling, the system admits a variational formulation. We prove existence of nonconstant steady states for small diffusion and uniqueness of the homogeneous state for large diffusion. Linear stability analysis shows that only unimodal patterns are stable, while multimodal solutions are unstable. Bifurcation analysis reveals subcritical and supercritical pitchforks, with fold bifurcations generating bistable regimes. Our results demonstrate that mechanochemical feedback provides a robust mechanism for single-peaked pattern formation without requiring a second diffusible inhibitor.
\end{abstract}

\maketitle

\section{Introduction}
\label{sec:Intro}
Self-organized pattern formation refers to the spontaneous emergence of spatially heterogeneous states from homogeneous equilibria in spatially extended systems governed by local interactions. In biological modeling, such phenomena are commonly described by coupled partial differential equations combining nonlinear reaction–diffusion dynamics with mechanical feedback, leading to symmetry breaking and spatial organization in developing or regenerating tissues. Despite extensive modeling efforts, a rigorous mathematical understanding of how nonlinear chemical kinetics interacting with mechanical couplings give rise to robust pattern selection remains incomplete \cite{Daphne2025,Halatek2018}. In particular, nonlocal mechanical constraints introduce qualitative changes in stability properties, mode selection, and bifurcation structure, posing significant analytical challenges for the theory of pattern formation.

Since Turing’s seminal proposal that diffusion-driven instabilities can generate spatial patterns \cite{turing1990chemical}, reaction-diffusion models have become a central theoretical framework for biological pattern formation \cite{doi:10.1126/science.1179047}. Canonical examples such as the Gierer-Meinhardt model \cite{GiererMeinhard1972} formalized the local activation--long-range inhibition (LALI) principle, showing how nonlinear interactions between an autocatalytic activator and a spatially extended inhibitory signal can give rise to stable patterns. In these classical Turing-type systems \cite{GiererMeinhardt2000, OsterMurray1989, Veerman2021}, pattern formation relies on at least two interacting morphogens with sufficiently distinct effective ranges, typically realized through differential diffusion. Despite their conceptual success, identifying concrete molecular implementations of Turing systems {\it in vivo} has remained notoriously difficult \cite{Marcon2012}. In particular, the nature of the long-range inhibitory signal is often unclear, raising the possibility that mechanisms beyond simple molecular diffusion may underlie observed patterns. Motivated by this, kernel-based formulations of pattern formation have been proposed, in which spatial interactions are encoded directly through nonlocal activation–inhibition kernels rather than diffusion operators \cite{KONDO2017120, Cygan2021}. Such models can reproduce the full repertoire of classical two-dimensional patterns, including spots, stripes, and networks, without explicitly invoking diffusive transport. More generally, it was demonstrated that reaction–diffusion systems with arbitrarily many components can be reduced to integro-differential descriptions characterized by effective kernels \cite{EI2021110496}, implying that distinct molecular mechanisms may generate identical spatial patterns if they share the same interaction structure. While powerful, this abstraction also lacks direct biological interpretation, as the cellular or molecular origins of the effective kernels are often not directly identifiable.

Growing experimental evidence indicates that mechanical forces are integral to biological organization across scales, challenging purely chemical views of pattern formation \cite{https://doi.org/10.15252/embr.202357739}. This has motivated mechanochemical frameworks in which tissue mechanics and biochemical signaling are dynamically coupled. These questions are often investigated using biologically motivated mathematical models, with patterning in the freshwater polyp {\it Hydra} providing a prominent example \cite{Vogg2019, NARAYANASWAMY2025204024}. Existing {\it Hydra} models range from classical reaction-diffusion descriptions of activation–inhibition dynamics \cite{GiererMeinhard1972, meinhardt2012modeling} to more recent formulations based on receptor-mediated interactions coupled to intracellular signaling \cite{doi:10.1142/S1793524524501547, Mercker2021.09.13.460125}. However, purely biochemical models appear insufficient to account for observed patterning behavior in {\it Hydra}. Both experimental and theoretical studies indicate that mechanical effects can act as effective long-range interactions and generate robust spatial organization \cite{Mercker2013, Hiscock2015, Ferenc2021}. Mechanochemical models naturally give rise to feedback loops between deformation and signaling \cite{Weevers2025}, suggesting that mechanical coupling can replace the classical diffusible inhibitor and provide an alternative realization of the LALI mechanism \cite{wang2023,Ferenc2021,BraunKeren2018}.

Several theoretical frameworks have incorporated mechanical effects into pattern formation models, ranging from simple mechanochemical feedbacks that generate coupled chemical–mechanical patterns \cite{Mercker2013, Brinkmann2018} to more elaborate continuum descriptions, such as biphasic poroelastic tissue models, in which mechanically induced flows enable pattern formation even for minimal biochemical kinetics \cite{doi:10.1073/pnas.1813255116}. Despite this progress, a systematic mathematical understanding of pattern selection and stability in mechanochemical systems remains incomplete. In contrast to classical Turing models, where mode selection and stability are well characterized, the influence of mechanical coupling on pattern multiplicity, robustness, and minimal pattern-forming requirements is still poorly understood.

In this work, we address these gaps by analyzing the stability and bifurcation structure of stationary solutions in a new mechanochemical model of symmetry breaking in {\it Hydra} aggregates, which was recently validated against experimental observations \cite{Weevers2025}. The model couples morphogen dynamics to tissue mechanics via a positive feedback mechanism: mechanical stretching enhances morphogen production, while morphogen concentration modulates tissue elasticity. Together with a global strain conservation constraint, this coupling yields a mechanochemical variant of the LALI paradigm that does not require a second diffusible inhibitor. We focus on an exponential elasticity–morphogen coupling, which endows the system with a variational structure amenable to rigorous analysis.

By deriving a reduced one-dimensional model that preserves the essential mechanochemical feedback while simplifying the geometric setting (Section \ref{sec:Deriv}, Appendix \ref{appsec:Deriv}), we obtain, for the first time, a framework amenable to rigorous analytical investigation of the dynamics. We establish the existence of spatially heterogeneous steady states for sufficiently small diffusion coefficients and prove uniqueness of the homogeneous equilibrium for large diffusion, deriving explicit bounds on the corresponding critical values by variational methods (Section \ref{sec:Existence_steady_states}). This analysis explains the mechanism of symmetry breaking, demonstrating how patterned states bifurcate from the homogeneous configuration.
We then investigate the pattern selection problem (Section \ref{sec:Stability}) and show that only unimodal stationary patterns are linearly stable, while all multimodal solutions are unstable, revealing a fundamental distinction from classical Turing systems and curvature-based mechanochemical models \cite{Veerman2021,Daphne2025}. Finally, we identify both subcritical and supercritical pitchfork bifurcations and reveal additional fold bifurcations responsible for the emergence of bistable regimes (Section \ref{sec:Bifurcation_analysis}). Throughout, the theoretical results are complemented by numerical simulations, elucidating the interplay between nonlocality, mechanical feedback, and nonlinear kinetics, and providing a mathematically grounded framework for experimental exploration.

\subsection*{Notation} We use the following notation: $\mathbb{{T}}$ is a one dimensional torus of size one, i.e., the unit interval $[0,1]$ in which we identify the endpoints with each another. Here 
$L^2(\T)$ denotes the set of square integrable functions on the circle defined via Fourier series
\begin{AlignedEquation}
    L^2(\T) = \left\lbrace u = \sum_{n=-\infty}^\infty \hat u(n) e^{2\pi i n x}: \sum_{n=-\infty}^\infty |\hat u(n)|^2 < \infty \right\rbrace
\end{AlignedEquation}
Similarly we denote Sobolev space  
\begin{AlignedEquation}
    W^{k,2}(\T) = \left\lbrace u \in L^2(\T): \sum_{n=-\infty}^\infty (1+n^2)^k |\hat u(n)|^2 < \infty \right\rbrace.
\end{AlignedEquation}
Constants are always denoted by a letter C even if they vary from line to line.

\section{Main results}
\label{sec:mainres}
We start by formulating  main theoretical results of this paper. We focus  the mechanochemical pattern formation model describing the concentration of the diffusive signaling molecules $u:= u(x,t)$, given by
\begin{AlignedEquation}
	\label{equ_varphi}
	\partial_t u(x,t) = D \, \partial_{xx} u(x,t) - u(x,t) + \kappa \frac{e^{u(x,t)}}{ \intT e^{u(y,t)} \, \dy}, \qquad  x \in [0,1], \quad t\in[0,\infty),
\end{AlignedEquation} 
with the coefficients $D>0$ and $\kappa >0$, supplemented with the periodic boundary conditions $u(0,t) = u(1,t)$ and $\partial_x u(0,t) = \partial_x u(1,t)$ for all $t\geqslant 0$. This model arises from a one-dimensional formulation of a hyperelastic material governed by the Saint-Venant–Kirchhoff constitutive law. A detailed derivation is provided in Section~\ref{sec:Deriv} and Appendix \ref{appsec:Deriv}. {Numerical simulations of the model (Figure~\ref{fig:numerics}) demonstrate its capacity to produce stable unimodal patterns.}
%\begin{figure}[h]
%    \centering
%\includegraphics[width=0.5\linewidth]{Graphs/dynamics_for_paper1.png}
 %   \caption{Numerical simulation of \eqref{equ_varphi} for $D=0.01$, $\kappa=1.5$ exhibits spontaneous pattern formation close to the constant steady state $U=\kappa$ and convergence to a nonconstant profile.}
 %   \label{fig:dynamics}
%\end{figure}

Global existence and uniqueness are classical results, established by standard arguments (see, e.g., \cite{miyasita2007dynamical}); we recall them here for completeness, 
\begin{proposition}[Global existence and uniqueness]
    \label{prop:ExistenceGlobal}
    For any initial datum $u_0 \in W^{1,2}(\mathbb T)$ problem~\eqref{equ_varphi} admits a unique global in time solution $u = u(x,t)$ satisfying
    \begin{AlignedEquation}
        u \in C\big([0,\infty), W^{1,2}(\T) \big), \quad u\in L^2\big([0,\infty), L^2(\T)\big), \quad u \in L^2\big((0,T), W^{2,2}(\T) \big) \quad \text{for each } T>0 \; .
    \end{AlignedEquation}
\end{proposition}
A complete proof is provided in Appendix for the reader’s convenience.

%\RemD{We should show graphs of such uni-modal stationary states for various values of $\kappa$ and $D$}

\begin{figure}[htbp]
\centering
\begin{minipage}{0.32\linewidth}
    \centering
    \includegraphics[width=\linewidth]{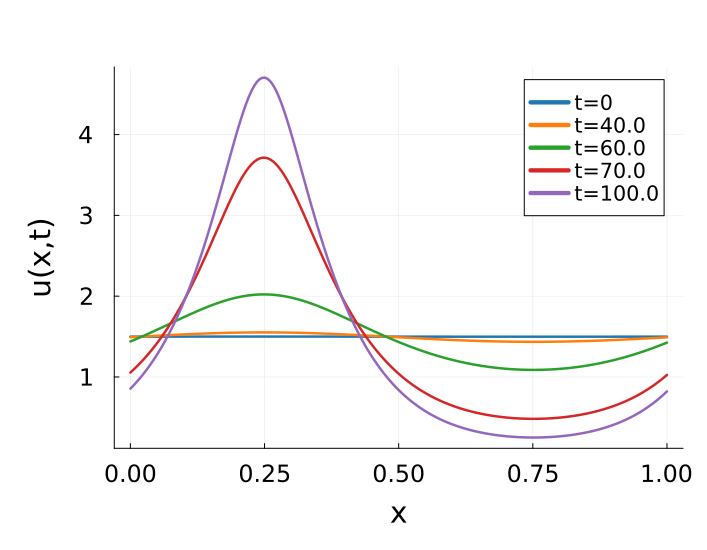}
\end{minipage}
\hfill
\begin{minipage}{0.32\linewidth}
    \centering
    \includegraphics[width=\linewidth]{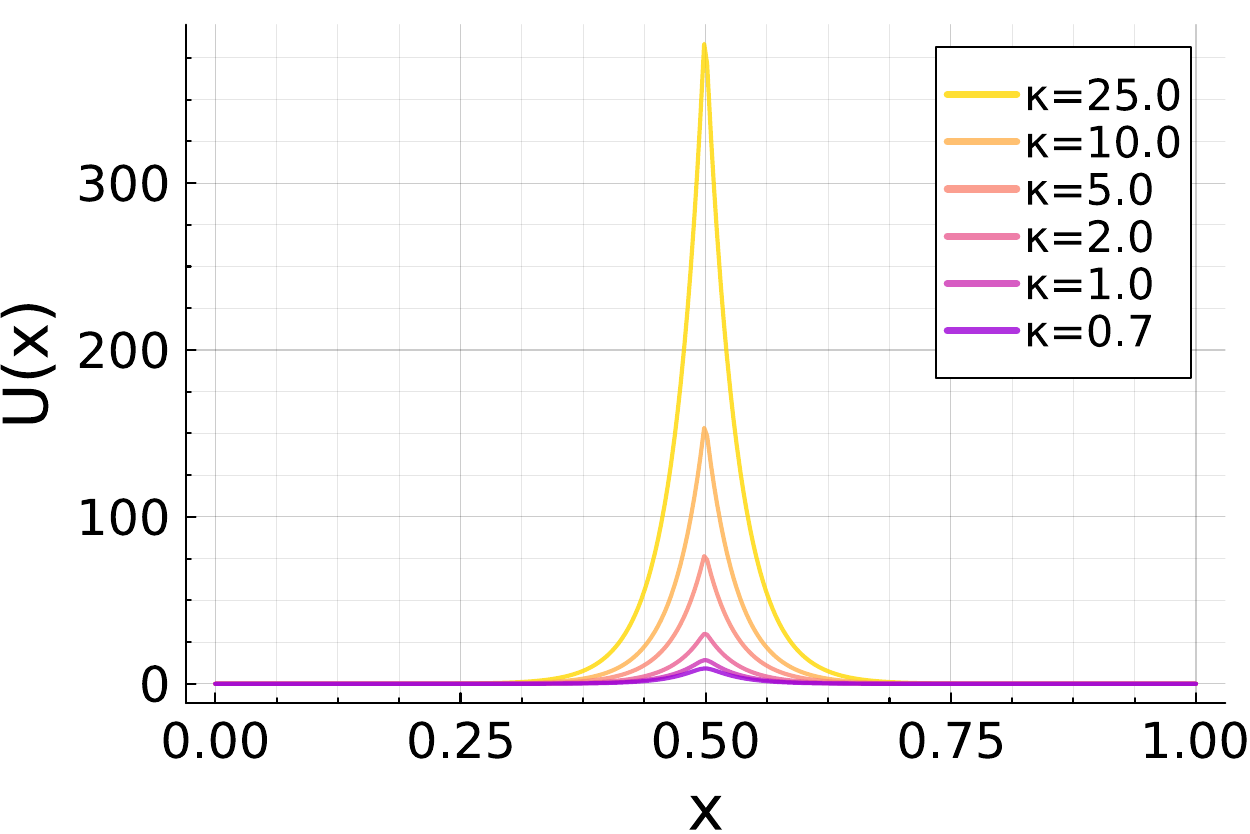}
\end{minipage}
\hfill
\begin{minipage}{0.32\linewidth}
    \centering
    \includegraphics[width=\linewidth]{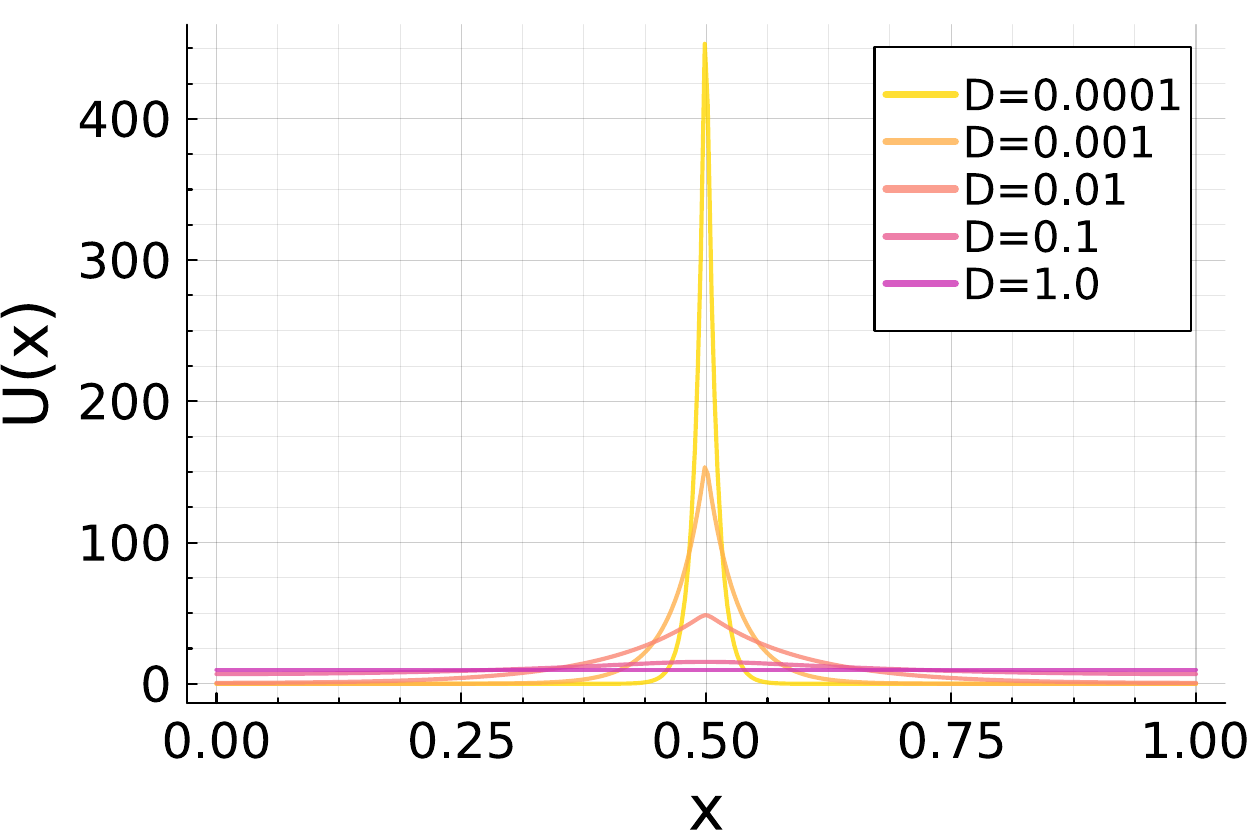}
\end{minipage}

\caption{
Numerical simulations of model \eqref{equ_varphi}. Left panel: Emergence of a uni-modal pattern from a small perturbation of the constant stationary solution $U=\kappa$, (example for $D=0.01$, $\kappa=1.5$). Middle panel: Nonconstant stationary solutions for fixed diffusion $D = 10^{-3}$ and varying~$\kappa$.
Right panel: Nonconstant stationary solutions for fixed $\kappa = 3.0$ and varying $D$. We observe that smaller diffusion leads to more concentrated patterns. 
}
\label{fig:numerics}
\end{figure}

%\begin{definition}[Weak solution]
%    \label{def:WeakSol}
  %  A function $U \in W^{1,2}(\T)$ is a weak solution to equation \eqref{equ_varphi_Stat} if for every $\varphi \in W^{1,2}(\T)$ there holds
  %  \begin{AlignedEquation}
   %     \label{eq:WeakSol}
 %       0 = -D\intT U_x \varphi_x dx -\intT U \varphi \, dx + \kappa \, \frac{\intT e^U \varphi \, dx}{\intT e^U \dy }. 
 %   \end{AlignedEquation}
%\end{definition}
We now characterize stationary solutions $u(x,t) = U(x)$ of \eqref{equ_varphi} (see Figure~\ref{fig:numerics}), which satisfy 
\begin{AlignedEquation}
	\label{equ_varphi_Stat}
	0 = D \,   U_{xx} - U + \kappa \frac{e^U}{ \intT e^{U}\dy}  \; ,
\end{AlignedEquation}
supplemented with the periodic boundary conditions $U(0) = U(1)$ and $U_x(0) = U_x(1)$. 
\begin{remark}[Mass conservation of stationary structures]
    \label{Rem:Cst}
   Integration of \eqref{equ_varphi} over $\mathbb{T}$ yields the mass evolution equation
\begin{AlignedEquation}
    \label{eq:MassDynamic}
    \frac{d}{dt} \intT u \, dx=\kappa - \intT u \, dx \; , 
\end{AlignedEquation}
implying $\int_\mathbb{T} u \, dx \to \kappa$ as $t \to \infty$. Thus, any stationary solution $U(x)$ satisfies $\int_\mathbb{T} U(x) \, dx = \kappa$. 
\end{remark}
\begin{theorem}[Existence of stationary solutions]    \label{thm:ExistenceStationary}
    For each $\kappa >0$ there exist two constants $D_{min}(\kappa)$ and $D_{max}(\kappa)$ satisfying  
    \begin{AlignedEquation}
       0 < D_{min}(\kappa) <  D_{max}(\kappa),
    \end{AlignedEquation}
    such that:
    \begin{itemize}
        \item If $0<D<D_{min}$, then problem \eqref{equ_varphi_Stat} admits a
        %weak
        nonconstant solution, which is Lyapunov stable with respect to $L^2$-norm for perturbations in $W^{1,2}(\T)$.
        \item If $D>D_{max}$, the constant function $\overline U \equiv \kappa$ is the unique solution to equation \eqref{equ_varphi_Stat} and globally asymptotically stable with respect to $L^2$-norm for perturbations in $W^{1,2}(\T)$.
    \end{itemize}
\end{theorem}
The proof, which employs variational methods, is presented in Section~\ref{sec:Existence_steady_states}.
We emphasize that the analytically derived bounds $D_{\min}$ and $D_{\max}$ are not optimal, as they arise from relatively coarse estimates; their refinement by numerical simulations is presented at the end of the section.

We classify nonconstant stationary solutions according to their spatial periodicity: an $m$-modal solution has $m$ peaks per period. In particular, solutions with $m=1$ are called {unimodal}, while those with $m=2,3,\dots$ are {multimodal}. By rescaling $x \mapsto m x$, a unimodal solution $U$ of equation~ \eqref{equ_varphi_Stat} yields an m-modal solution $\tilde U(x)=U(mx-\lfloor mx \rfloor)$ of 
%\RemD{I think that here, $U$ is the multi-modal solution and what is currently called $U_m$ is the uni-modal one, please check!}
\begin{AlignedEquation}
    \label{eq:StationaryM}
    0 = \frac{D}{m^2} \tilde U_{xx} -  \tilde U + \kappa\frac{e^{\tilde U}}{\I e^{\tilde U} \dy } \; ,
\end{AlignedEquation}
with effective diffusivity $D/m^2$.

To investigate the stability properties of stationary solutions of \eqref{equ_varphi}, we consider perturbations of the form
\[
u(x,t) = U(x) + e^{\lambda t} \varphi(x),
\]
where $U$ is a stationary solution. Substituting this ansatz into \eqref{equ_varphi} leads to the following nonlocal eigenvalue problem:
\begin{AlignedEquation}
    \label{eq:EigProblem}
    D \varphi_{xx} + \left(\kappa\frac{e^U}{\intT e^U \dy} - 1 - \ENS\right) \varphi = \kappa \frac{e^U}{\left(\intT e^U \dy\right)^2} \intT e^U \varphi \, \dy.
\end{AlignedEquation}
We denote by $\ENS$ an eigenvalue of the nonlocal problem and compare it below with the eigenvalues $\ELS$ of the corresponding local problem. The following results describe the linear stability of nonconstant stationary solutions.
\begin{theorem}[Type of stationary solutions]
    \label{thm:StabilityNCst}
    The stationary solutions are classified as follows:
    \begin{itemize}
        \item The stable nonconstant solution obtained in Theorem~\ref{thm:ExistenceStationary} is unimodal.
        \item All multimodal solutions ($m$-modal with $m \geq 2$) are linearly and nonlinearly unstable in $L^2(\mathbb{T})$.
    \end{itemize}
\end{theorem}
The proof, based on the gradient flow structure of \eqref{equ_varphi} and Sturm–Liouville theory, is presented in detail in Section~\ref{sec:Stability}. 

\begin{corollary}[Existence of multimodal solutions]
    \label{thm:ExistenceStationaryMultimodal}
    Combining Theorems~\ref{thm:ExistenceStationary} and \ref{thm:StabilityNCst} with equation~\eqref{eq:StationaryM}, we obtain:
    \begin{itemize}
        \item If $D < D_{\min}/m^2$, then an $m$-modal solution to \eqref{equ_varphi_Stat} exists.
        \item If $D > D_{\max}/m^2$, then no $m$-modal weak solutions exist.
    \end{itemize}
\end{corollary}
    This result highlights a fundamental principle: stable pattern formation in our mechanochemical framework necessarily yields unimodal structures, whereas all multimodal configurations are intrinsically unstable. This represents a significant departure from classical Turing models, in which the number of peaks corresponds to the unstable eigenmodes and can vary with domain size~\cite{Murray}. Moreover, pattern selection in our model differs from other mechanochemical frameworks that permit multiple, or even arbitrarily many, peaks, see, e.g., \cite{Daphne2025, Mercker2013, Brinkmann2018}.

% The transition from spatially homogeneous to patterned states represents a fundamental mechanism in mechanochemical pattern formation. We conduct a comprehensive bifurcation analysis to characterize how nonconstant steady states emerge from constant solutions as system parameters vary, revealing the precise conditions under which spatial symmetry breaking occurs. This analysis determines both the existence and stability properties of patterned states, which directly influence the robustness of biological pattern formation. Our investigation employs the Crandall-Rabinowitz theorem to provide local and global characterizations of bifurcating solution branches.

We complete the analysis by characterizing the emergence of patterned states through bifurcations from the homogeneous steady state. This approach provides insight into emergence of patterns, which is the main focus when applying the model to biological systems. 
\begin{theorem}[Bifurcation structure]
\label{thm:BifurcationStructure}
For each $n \in \mathbb{N}$, define $\kappa_n := 1 + 4\pi^2 n^2 D$. Then the following holds:
\begin{enumerate}
\item \textbf{Existence.}
At each bifurcation point $(\bar{\Phi}_n,\kappa_n)$, with $\bar{\Phi}_n=\kappa_n$, there exists a smooth local branch of nonconstant solutions
\[
\Gamma_n(s) = \big(\bar\varphi_n(s),\kappa_n(s)\big),
\]
admitting the following asymptotic expansion as $s\to 0$,
\begin{AlignedEquation}
\bar\varphi_n(s) = \kappa_n + s\sqrt{2}\cos(2\pi n x) + \mathcal{O}(s^2),\
\kappa_n(s) = \kappa_n + C_n \, s^2 + \mathcal{O}(s^3),
\end{AlignedEquation}
where $C_n = \frac{1}{4} - \frac{1}{16 D n^2 \pi^2} + 2 D n^2 \pi^2$.
\item \textbf{Type.}
The bifurcation at $(\bar{\Phi}_n,\kappa_n)$ is
\begin{itemize}
\item subcritical if $\kappa_n < \tfrac{3}{2}$, equivalently $D < \frac{1}{8\pi^2 n^2}$,
\item supercritical if $\kappa_n > \tfrac{3}{2}$, equivalently $D > \frac{1}{8\pi^2 n^2}$.
\end{itemize}
\item \textbf{Stability.}
The branch $\Gamma_n(s)$ is linearly stable near the bifurcation point if and only if $n=1$ and $\kappa_1 > \tfrac{3}{2}$. All remaining branches are linearly unstable in a neighborhood of the bifurcation point.
 Moreover, this local branch admits a global continuation in the sense of Rabinowitz and
\item \textbf{Turning point.}
 If $\kappa_1 < \tfrac{3}{2}$, the branch $\Gamma_1(s)$ undergoes a fold bifurcation at some $\kappa_f \in (0,\kappa_1)$.
\end{enumerate}
\end{theorem}
Theorem~\ref{thm:BifurcationStructure} summarizes the bifurcation analysis carried out in Section~\ref{sec:Bifurcation_analysis}. It combines local bifurcation results obtained via the Crandall–Rabinowitz theorem with a global continuation analysis in the sense of Rabinowitz, thereby revealing the presence of fold bifurcations and bistable regimes. Together, these results show that the transition from the homogeneous to the patterned state is mediated by a symmetry-breaking bifurcation whose direction and stability depend critically on the balance between diffusion and reaction strength. In particular, the bifurcation undergoes a change of type at a critical parameter threshold, a phenomenon also observed in other nonlinear pattern-forming systems \cite{Berlyand2025_changeBifType}. They also provide a mathematical mechanism for the emergence of bistability, corresponding to the coexistence of homogeneous and patterned steady states. Such bistable behavior has been recently reported in {\it Hydra} regeneration \cite{Tursch2022}.

\begin{figure}[hb]
    \centering
\includegraphics[width=0.8\linewidth]{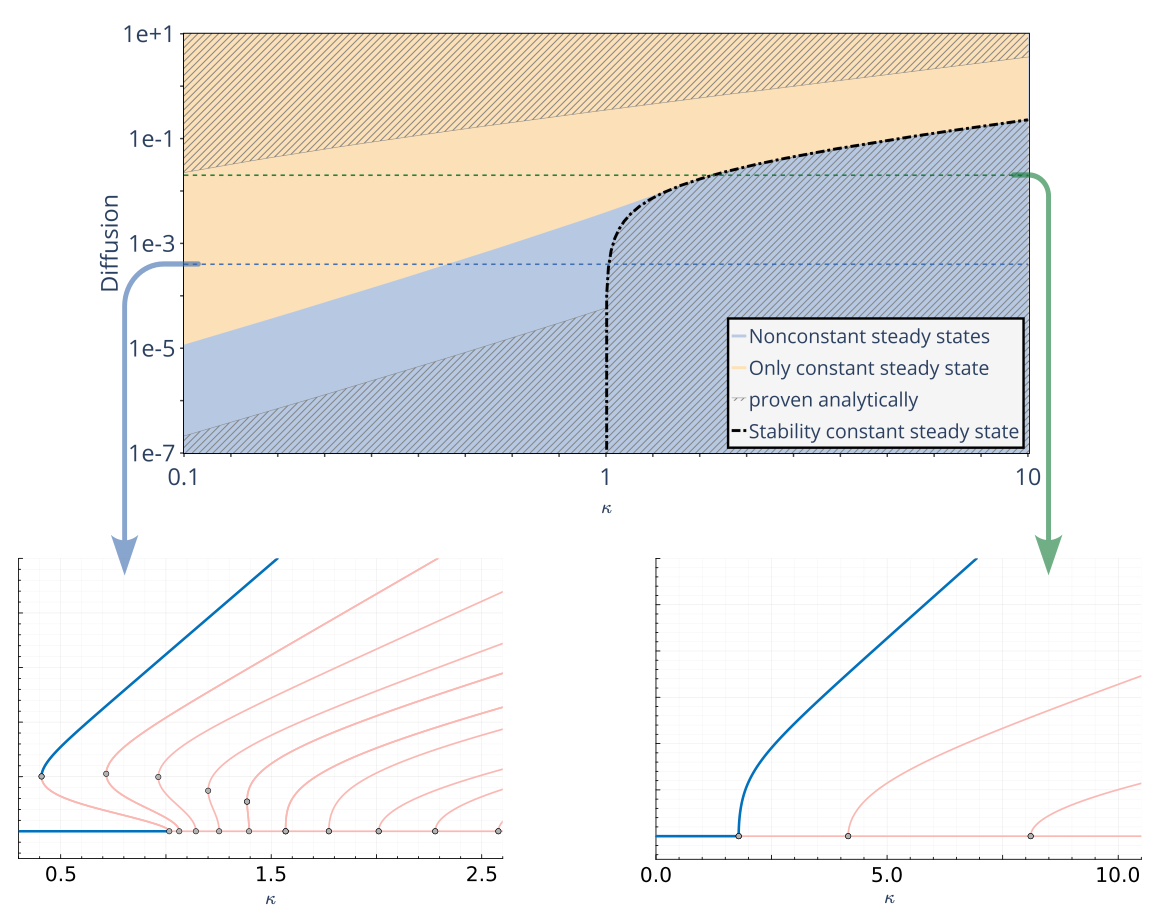}
    \caption{Parameter space analysis and bifurcation structure.
(Top) Parameter space $(D, \kappa)$ showing regions of numerical convergence to nonconstant (blue) versus constant (yellow) steady states, with gray areas indicating analytically proven results. (Bottom) Bifurcation diagrams demonstrating subcritical regime with bistability (left, $D < 1/(8\pi^2)$) and supercritical regime with direct pattern formation (right, $D > 1/(8\pi^2)$), where blue and red curves denote stable and unstable branches, respectively.}
    \label{fig:ParameterSpace}
    \vspace{-0.2cm}
\end{figure}

We conclude this section by illustrating our analytical results through numerical computations. The analytical bounds $D_{\min}(\kappa)$ and $D_{\max}(\kappa)$ obtained in Theorem \ref{thm:ExistenceStationary} are not expected to be sharp, as they arise from relatively coarse estimates. Nevertheless, the numerical results presented below demonstrate that these bounds correctly capture the qualitative structure of the parameter space. In particular, Figure~\ref{fig:ParameterSpace} confirms the predicted transition between subcritical and supercritical bifurcation regimes at the critical threshold. Moreover, the numerical simulations illustrate the existence of bistable parameter regions arising from fold bifurcations and thereby support the overall bifurcation framework developed in the analytical part of this work. While Theorem \ref{thm:BifurcationStructure} provides a complete characterization of bifurcating branches in the neighborhood of each bifurcation point, the global continuation of these branches lies beyond the reach of the analytical techniques employed here. In the subcritical regime $\kappa < 1.5$ the theory guarantees only that the primary branch $\Gamma_1(s)$ possesses at least one turning point. Our numerical analysis indicates, however, that this branch exhibits exactly one fold bifurcation in the interval $(0, \kappa_1)$, and that the resulting stable branch does not undergo further secondary bifurcations.\\
These observations complement Theorem~\ref{thm:ExistenceStationary}, which ensures existence of a stable nonconstant stationary solution. Numerically, we observe precisely one stable nonconstant solution branch, unique up to spatial shifts due to periodic boundary conditions, suggesting uniqueness of the variational minimizer modulo translations.
In the context of the model application, Lyapunov stability is sufficient, as it describes symmetry breaking in systems where the position of the emerging organizer (given by a peak in morphogen concentration) is not prescribed {\it a priori}. 
 In addition, all unstable branches emerging from bifurcation points $(\bar{\Phi}_n, \kappa_n)$ with $\kappa_n < 1.5$ display an analogous subcritical structure with a single fold bifurcation. Branches associated with different modal numbers remain disjoint for all values of $\kappa$.

\section{Model formulation}
\label{sec:Deriv}

\begin{figure}[b]
\centering
\includegraphics[width=0.6\linewidth]{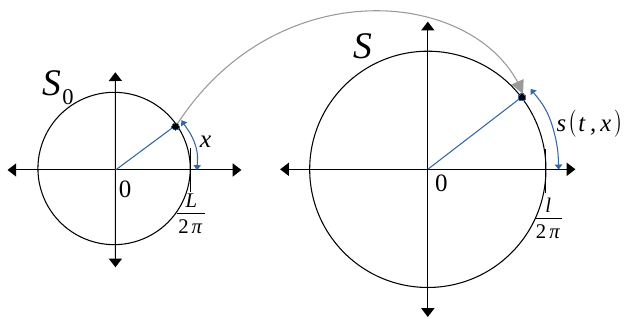}
\caption{Schematic representation of the elastic deformation framework for the {\it Hydra} spheroid model. Left: Circular reference configuration $\mathcal{S}_0$ parametrized by arc-length $x \in [0,L]$.  Right: Deformed circular configuration $\mathcal{S}$ in which material points have arc-length position $s(x) \in [0, \ell]$. }
\label{fig:sketch2d_frameworkold}
\end{figure}

We reformulate the mechanochemical patterning model for \emph{Hydra} spheroids introduced in~\cite{Weevers2025} in two spatial dimensions in which the spheroid is described as a closed curve.
Starting from the model derivation and formulation in 3D using index notation, one can specialise the notation in 2D as it was done in the supplementary material of ~\cite{Weevers2025}. Alternatively, we derive the model directly, using the same modelling arguments but without using index notation, in the Appendix~\ref{appsec:Deriv}.

Ultimately, we derive a reduced spatially one-dimensional formulation of the original model in~\cite{Weevers2025} that retains the essential coupling between mechanics and chemistry while simplifying the geometric structure. In this setting, the \emph{Hydra} spheroid is represented as an elastic closed curve, providing a minimal framework to analyze the core mechanochemical feedback mechanisms, see Figure~\ref{fig:sketch2d_frameworkold}.

The stress-free reference configuration, denoted by $\mathcal S_0$, is assumed to be a circle of fixed circumference $L>0$. It is parametrised by $x \in [0, L]$ which denotes the arc-length associated to a given point in $\mathcal{S}_0$. We prescribe a fixed enclosed area $A_0>\operatorname{Area}(\mathcal S_0)$ and consider deformations of the reference shape which enclose the area $A_0$. In Appendix~\ref{appsec:Deriv} we show that, as a consequence of stress balance and the Law of Laplace, the deformed spheroid is always a circle of the given circumference $\ell>L$ and that no other shapes are possible.

The scalar function $s(x):[0,L]\mapsto[0,l]$, $s'>0$, represents the arc-length position in the deformed spheroid $\mathcal{S}$ of the material point with arc-length position $x$ in $\mathcal{S}_0$. As the deformed spheroid is circular, any deformation is by shifting material points tangentially within the hydra tissue. The full mechanochemical system written in terms of the morphogen concentration $u=u(t,x)$ and the strain $\varepsilon(t,x)$ reads
\begin{equation} \label{sys_eps}
\left\{
\begin{aligned}
    &\begin{aligned}        
        \partial_t u &= \kappa \, \varepsilon - \alpha u + D \, \partial_{xx} u,\\
        0 &= \partial_x \big(E(u) \, \varepsilon\big),
    \end{aligned} 
    && x \in [0,L], \quad t \ge 0,\\
    & l - L = \int_0^L \varepsilon \, d x, && t \ge 0.
\end{aligned}
\right.    
\end{equation}
This system of equations combines stress balance where $E(u)>0$ is the elastic modulus given by a monotonically decreasing function of $u$. This models increased tissue elasticity (reduced stiffness) in response to higher expression of morphogen. This implies that solutions {are} characterised by larger strain where the morphogen concentration is high.

This ''mechanical'' equation is coupled to a ''chemical'' reaction diffusion equation for the morphogen concentration. Here, $D$ denotes diffusivity, $\alpha$ is the spontaneous disassociation rate and $\kappa>0$ is a given parameter associated to the expression of morphogen in response to tissue strain $\varepsilon$. Up-regulation of morphogen concentration in response to tissue strain complements increased strain where morphogen is over-expressed according to the mechanical part of the model. Together, these two components form a positive  local feedback loop.

Note that following infinitesimal stress theory, we write the strain as $\varepsilon=s'-1$ which motivates the third equation in \eqref{sys_eps} and expresses conservation of total strain, effectively introducing mechanical long-range inhibition, which acts alongside the positive feedback between local strain and morphogen concentration. This combination implements the classical local activation - long-range inhibition (LALI) mechanism.\\
From the momentum balance equation in \eqref{sys_eps}, the stress depends only on time $t$, that is,
\begin{AlignedEquation}
    \sigma(t) := E(u(t,\xi)) \, \varepsilon(t,\xi).
\end{AlignedEquation}
In particular, the strain can be expressed as
\begin{AlignedEquation}
    \varepsilon(t,\xi) = \frac{\sigma(t)}{E(u(t,\xi))}.
\end{AlignedEquation}
Integrating the strain over the whole domain and using the total strain constraint gives
\begin{AlignedEquation}
    l - L = \int_0^L \varepsilon(t,\xi) \, d\xi 
           = \sigma(t) \int_0^L \frac{1}{E(u(t,\xi))} \, d\xi.
\end{AlignedEquation}
Solving for the stress, we obtain
\begin{AlignedEquation}
    \sigma(t) = \frac{l - L}{\int_0^L \frac{1}{E(u(t,\xi))} \, d\xi},
\end{AlignedEquation}
and the corresponding strain reads
\begin{AlignedEquation}
    \varepsilon(t,x) = \frac{\sigma(t)}{E(u(t,x))} 
                         = \frac{l-L}{E(u(t,x)) \int_0^L \frac{1}{E(u(t,y))} \, dy}.
\end{AlignedEquation}
Using the closed-form expression for the strain, the morphogen equation in \eqref{sys_eps} can be written as
\begin{AlignedEquation}
    \label{eq:BasicLong}
    \partial_t u(t,x) = \frac{\kappa (l-L)}{E(u(t,x)) \int_0^L \frac{1}{E(u(t,y))} \, dy} - \alpha u + D \, \partial_{xx} u \, .
\end{AlignedEquation}
For the remainder of the manuscript, we assume that the elastic modulus depends exponentially on the morphogen concentration, 
\begin{AlignedEquation}
    E(u) = e^{-\beta u}, \quad \beta > 0 \, .
\end{AlignedEquation}
This choice allows \eqref{eq:BasicLong} to be endowed with a variational structure and ensures that the linearization of the spatial operator is self-adjoint.\\
We next nondimensionalize \eqref{eq:BasicLong} by introducing the scaling
\begin{AlignedEquation}
    t \mapsto t_0 \, t, \qquad \xi \mapsto x_0 \, x, \qquad u \mapsto u_0 \, u,
\end{AlignedEquation}
which leads to
\begin{AlignedEquation}
    \frac{u_0}{t_0} \, u_t = \kappa \frac{(l-L) e^{\beta u_0 u}}{x_0 \int_0^{L/x_0} e^{\beta u_0 u} \, dx} - \alpha u_0 u + D \frac{u_0}{x_0^2} \, u_{xx}.
\end{AlignedEquation}
Choosing 
\begin{AlignedEquation}
    u_0 = \frac{1}{\beta}, \quad t_0 = \frac{1}{\alpha}, \quad x_0 = L,
\end{AlignedEquation}
and defining the dimensionless constants
\begin{AlignedEquation}
    \kappa \frac{\beta (l-L)}{\alpha L} \mapsto \kappa, \qquad \frac{D}{\alpha L^2} \mapsto D,
\end{AlignedEquation}
we obtain the final nondimensionalized system
\begin{AlignedEquation}
    \label{eq:FinalDimless}
    u_t = D \, u_{xx} - u + \kappa \frac{e^u}{\int_0^1 e^u \, dy},
\end{AlignedEquation}
for $x \in [0,1]$, supplemented with periodic boundary conditions
\begin{AlignedEquation}
    u(0,t) = u(1,t), \qquad \partial_x u(0,t) = \partial_x u(1,t), \quad t \ge 0.
\end{AlignedEquation}

\section{Stationary solutions}\label{sec:Existence_steady_states}
\subsection{Variational formulation} Equation~\eqref{equ_varphi} can be interpreted as a gradient flow of the free energy functional with respect to $L^2$-norm (see \textit{e.g.} \cite{ambrosio2005gradient}), with
\begin{AlignedEquation}
    \label{eq:FunDef}
        \mathcal{J}\big(u(\cdot, t)\big) =  \frac{D}{2}\intT\big|u_x(x,t)\big|^2 \, dx + \frac{1}{2} \intT \big|u(x,t)\big|^2 \, dx - \kappa \log\left(\intT e^{u(x,t)} \, dx\right).
\end{AlignedEquation}
% and domain $L^2((0,T),W^{1,2}(\T))$. % TODO: add domain?
The first and second variations of $\mathcal{J}$ are given by
\begin{AlignedEquation}
        \label{eq:Ders}
        D\mathcal{J}\big(u(\cdot , t)\big)\psi&= 
	 D \intT u_x(x,t)\, \psi_x \, dx +\intT u(x,t)\, \psi \, dx-
	\kappa \frac{ \intT e^{u(x,t)} \, \psi \, dx }{ \intT e^{u(x,t)} \, dx }, \\
	D^2\mathcal{J}\big(u(\cdot, t)\big)(\psi,\xi)&=  
	D \intT \psi_x\, \xi_x \, dx +\intT \psi \, \xi \, dx \\
    &\quad-
	\kappa \left(  \frac{ \intT e^{u(x,t)} \, \psi \, \xi \, dx }{ \intT e^{u(x,t)} \, dx }-
  \frac{ \intT e^{u(x,t)} \, \psi \, dx \intT e^{u(x,t)} \, \xi \, dx }{ \left( \intT e^{u(x,t)} \, dx\right)^2 }\right) \; .
\end{AlignedEquation}
\\
The focus of this section is to analyze solutions to equation \eqref{equ_varphi_Stat} through the critical points of the functional $\mathcal{J}: W^{1,2}(\mathbb{T}) \to \mathbb{R}$, defined by formula \eqref{eq:FunDef}. A direct analysis of this functional can be challenging due to the difficulty in establishing the existence of nonconstant minimizers. To overcome these difficulties, we decompose the function $\varphi \in W^{1,2}(\mathbb{T})$ using the orthonormal eigenbasis of the Laplace operator with periodic boundary conditions.

\begin{proposition}
    \label{prop:Spectral}
    There exist $\lbrace \FL_k \rbrace_{k=0}^\infty$ an orthonormal basis of $L^2(\T)$ and an orthogonal basis of $W^{1,2}(\T)$ and a set of eigenvalues $\lbrace\EL_k \rbrace_{k=0}^\infty \subseteq [0, \; \infty)$, such that $(\FL_k, \EL_k)$ is the eigenpair of the operator~$\Delta$
    \begin{AlignedEquation}
        \Delta \FL_k = -\EL_k \FL_k,
    \end{AlignedEquation}
    with periodic boundary conditions. 
\end{proposition}

\begin{remark}
    For $\T = [0,1]$ with periodic boundary conditions, the eigenpairs of the Laplacian are given by
    \begin{AlignedEquation}
        \lbrace \FL_k \rbrace_{k=0}^\infty &= \big\lbrace 1 \big\rbrace \bigcup\big\lbrace \sqrt{2} \sin(2k\pi x): k\in \mathbb{N} \big\rbrace \bigcup\big\lbrace \sqrt{2}\cos(2k\pi x): k\in \mathbb{N_+} \big\rbrace, \\
        \lbrace\EL_k \rbrace_{k=0}^\infty &= \lbrace 4\pi^2 k^2: k\in \mathbb{N} \rbrace
    \end{AlignedEquation}
    These eigenfunctions are uniformly bounded in $L^\infty$-norm, with $\|\FL_k\|_\infty = \sqrt{2}$ for $k \geq 1$ and $\|\FL_0\|_\infty = 1$.
\end{remark}

\noindent Using the eigenfunction decomposition of the Laplacian, we derive an equivalent system of algebraic equations satisfied by the Fourier coefficients of a stationary solution. Expressing the solution $U(x)$ to equation \eqref{equ_varphi_Stat} in the eigenbasis as $U = \sum_{k=0}^\infty \overline{R}_k \FL_k$, with $\overline{R} \in \ell^2(\N)$, we obtain 
\begin{AlignedEquation}    
    \label{eq:StatBas}
    0 = -D\SI{k} \overline R_k\EL_k \FL_k -  \SI k \overline R_k \FL_k + \kappa \frac{e^{\overline R_0 +  \SIO k \FRD k}}{\IE{\overline R_0 +\SIO k \FRD k}}.
\end{AlignedEquation}
We multiply equation \eqref{eq:StatBas} by $\FL_j$ and integrate over  domain $\mathbb{T}$. Then, dividing both the numerator and denominator by $e^{\overline{R}_0}$ and utilizing the orthonormality of the eigenbasis $\{\FL_j\}_{j=0}^\infty$, we obtain the following system of equations:
\begin{AlignedEquation}
    \label{eq:Rsystem0}
    0 = -\overline R_j (D\EL_j + 1) + \kappa \frac{\I{ e^{\SIO k \FRD k}\FL_j}}{\IE{\SIO k \FRD k}}, \qquad \text{for} \quad j\in \N
\end{AlignedEquation}
which we rewrite in the form
\begin{AlignedEquation}
    \label{eq:Rsystem}
    \overline R_j = \frac{\kappa}{1 + D\EL_j} \frac{\I{ e^{\SIO k \FRD k}\FL_j}}{\IE{\SIO k \FRD k}}, \qquad \text{for} \quad j\in \N.
\end{AlignedEquation}

\begin{remark}
    \label{rem:Rjbound}
    The right-hand side of equation \eqref{eq:Rsystem} is bounded and converges to $0$ as $j \to \infty$.  Indeed, we have $\|\exp(\sum \overline{R}_k \FL_k) \FL_j\|\Lo \leqslant \|\FL_j\|\Li\|\exp(\sum \overline{R}_k \FL_k)\|\Lo$. Since the eigenfunctions are uniformly bounded with $\|\FL_j\|_\infty = \sqrt{2}$ it follows that
    \begin{AlignedEquation}
        |\overline{R}_j| \leqslant \frac{\sqrt 2 \kappa }{1 + D\EL_j} \xrightarrow{j\to \infty} 0.
    \end{AlignedEquation}
\end{remark}

\noindent To facilitate the analysis of nonconstant solutions to  problem \eqref{equ_varphi_Stat}, we introduce a discrete counterpart of the functional $\mathcal J$ given by formula \eqref{eq:FunDef}, which offers a more convenient framework for establishing the existence of nonconstant solutions. We express the function $\varphi(x) \in W^{1,2}(\T)$ in the eigenbasis, then it holds
\begin{AlignedEquation}
    \label{eq:DiscreteFunctional}
    \mathcal J(\varphi) = \frac{1}{2}R_0^2 - \kappa R_0 +  \SIO k \frac{R_k^2}{2} (1 + D\EL_k ) - \kappa \log \left( \IE {\SIO k \FR k} dx \right),
\end{AlignedEquation}
where $\varphi(x) = \sum_{k=0}^\infty R_k \FL_k(x)$ with $R\in\omega_\EL^2$ and 
\begin{AlignedEquation}
    \label{eq:Domain}
    \omega^2_\EL = \left\lbrace {R} = (R_0, \, R_1, \, \ldots) \in \ell^2(\N_0): \|R\|_{\omega_\EL^2} :=  \sum_{k=0}^\infty R_k^2 \, (1 + D\EL_k) < \infty  \right\rbrace.
\end{AlignedEquation}
The second derivative of $\mathcal J$ given by equation \eqref{eq:Ders} and expressed in the eigenbasis, takes the form
\begin{AlignedEquation}
    \label{eq:SecDer0}
    \mathcal J_{\FL_0\FL_0}(\varphi) = 1 \qquad \text{and} \qquad \mathcal J_{\FL_0\FL_i}(\varphi) = \mathcal J_{\FL_i\FL_0}(\varphi) = 0 \quad \text{for} \quad i \in\N_+,
\end{AlignedEquation}
together with
\begin{AlignedEquation}
    \label{eq:HesSecDer}
    \mathcal J_{\FL_i\FL_i}(\varphi) &= 1 + D\EL_i  - \kappa { \I{\FL_i^2}(x) }p(x) + \kappa\left({ \I{\FL_i(x) p(x)}}\right)^2, \\
    \mathcal J_{\FL_i\FL_j}(\varphi) &= - \kappa { \I{\FL_i(x)\FL_j(x)p(x)} } + \kappa{\left( \I{\FL_i(x)p(x)}\right)\left( \I{\FL_j(x)p(x)}\right)},
\end{AlignedEquation}
where $p(x) = e^{\varphi} /\IE{\varphi}$ defines a probability measure on $(0,1)$. Thus by $L^\infty$-bounds of the eigenfunctions and Jensen inequality we immediately obtain
\begin{AlignedEquation}
    \label{eq:CoefEstim}
    1 + D\EL_i - 2\kappa \leqslant J_{\FL_i \FL_i}(\phi) \leqslant 1 + D\EL_i \qquad \text{and}\qquad |J_{\FL_i \FL_j}(\phi)| \leqslant 4\kappa
\end{AlignedEquation}
for all $i,j\in\N_+$.

We recall two well established properties of the functional $\mathcal{J}$ defined in \eqref{eq:FunDef}. These are included for completeness of the exposition.
\begin{lemma}
    \label{thm:Coercive}
    The functional $\mathcal{J}:W^{1,2} (\T)\to \R$ is bounded from below and coercive.
\end{lemma}

\begin{proof}
The proof follows directly. By setting $C = \min(D, 1)/2$ and utilizing the embedding $W^{1,2}(\T) \subseteq L^\infty(\T)$, we obtain
    \begin{AlignedEquation}
        J(\varphi) &= \frac{D}{2}\intT (\varphi_x)^2 + \frac{1}{2}\intT \varphi^2 - \kappa \log\left(\intT e^\varphi\right) \\ &\geqslant C \|\varphi\|_{W^{1,2}}^2 - \kappa \|\varphi\|_{L^\infty} \\
        & \geqslant  C\|\varphi\|_{W^{1,2}} \left( \|\varphi\|_{W^{1,2}} - \kappa/C\right)
    \end{AlignedEquation} 
    and $J(\varphi) \to \infty$ as $\|\varphi\|_{W^{1,2}}\to \infty$.
\end{proof}

\begin{lemma}
    \label{thm:WeaklyLower}
    The functional $\mathcal{J}:W^{1,2} (\T)\to \R$ is weakly lower semicontinuous.
\end{lemma}

\begin{proof}
    Let $\varphi_n\in W^{1,2}(\T)$ be a sequence converging $\varphi_n \rightharpoonup \varphi$ weakly in $W^{1,2}$. Clearly
    \begin{AlignedEquation}
        \intT |\partial_x\varphi|^2 \, dy &\leqslant \liminf_{n\to\infty} \intT |\partial_x\varphi_n|^2 \, dy, \\
        \intT |\varphi|^2 \, dy &\leqslant \liminf_{n\to\infty} \intT |\varphi_n|^2 \, dy.
    \end{AlignedEquation}
    Moreover by compact embedding, the sequence $\varphi_n \to \varphi$ converges strongly in $L^2$ and hence by Egorov theorem for each $\varepsilon>0$ there exist $U\subseteq \T$ such that $|\T \setminus U| <\varepsilon$ and $\varphi_n(x) \to \varphi(x)$ uniformly on $U$. By Sobolev embedding the sequence $\varphi_n$ is uniformly $L^\infty$-bounded providing
    \begin{AlignedEquation}
        \log\left(\int e^{\varphi_n}\right) \to\log\left(\int e^{\varphi}\right) \quad \text{as} \quad n\to\infty
    \end{AlignedEquation}
    and hence $\mathcal{J(\varphi)} \leqslant \liminf \mathcal{J}(\varphi_n)$.
\end{proof}

\subsection{Geometry of the functional}

We examine the critical points of the energy functional 
$\mathcal{J}$ given by the formula \eqref{eq:FunDef}.  A constant solution $\overline U \equiv \kappa$ is a critical point of $\mathcal{J}$ with $\mathcal{J}(\overline{U}) = -\kappa^2/2$. We demonstrate that, under certain conditions, such a constant solution is not a global minimizer of $\mathcal J$.

\begin{lemma}
    \label{thm:ExistenceOfMinus}
    For each $\kappa >0$ there exist a constant $D_{min}(\kappa)$ such that for all $D<D_{min}(\kappa)$, the infimum of the functional $\mathcal J$ is strictly less than $-\kappa^2/2$.
\end{lemma}

\begin{proof}
    We proceed by considering two distinct cases. For an arbitrary $\kappa > 0$, we first construct a vector $\hat{R}^N \in \mathbb{R}^\N$ defined as
    \begin{AlignedEquation}
        \hat R^N = \left(\kappa, \ \frac{\sqrt{2} \kappa }{1+D\EL_1}, \ \ldots \ , \ \frac{\sqrt 2 \kappa}{1+D\EL_N}, \ 0, \ 0, \ \ldots \right).
    \end{AlignedEquation}
    This vector is a finite-dimensional approximation of the maximal possible values of $\overline R_j$, as noted in Remark \ref{rem:Rjbound}. Let us consider the corresponding function $\varphi_N(x)$ expressed in terms of the eigenbasis~$\lbrace \EL_k(x)\rbrace_{k=0}^\infty$, where the cosine components are given by $\hat R^N$ and the sine components vanish:
    \begin{AlignedEquation}
        \varphi_N(x) = \sum_{k=0}^N \hat R^N_k \hat \FL_k(x) \quad \text{with} \quad \hat \FL_k(x) = \sqrt{2} \cos(2\pi k x), \quad \hat \FL_0 = 1.
    \end{AlignedEquation}
    We demonstrate that $\mathcal{J}(\varphi_N) < -\kappa^2/2$ for sufficiently large $N$. From equation \eqref{eq:DiscreteFunctional}, we have
    \begin{AlignedEquation}
        \label{eq:FunPoint}
        \mathcal{J}\left( \varphi_N\right) = -\frac{\kappa^2}{2} + \kappa^2\sum_{k=1}^N \frac{1}{1+D\EL_k} - \kappa \log \left(\intT e^{\sum_{k=1}^N \frac{\sqrt 2\kappa}{1+D\EL_k}\hat \FL_k}dx\right).
    \end{AlignedEquation}
    Now we apply a rough estimate for the functions $\hat \FL_k(x) = \sqrt{2} \cos(2\pi k x)$. For each $k \in \{1, \ldots, N\}$ we have
    \begin{AlignedEquation}
       \hat \FL_k(x) \geqslant 1 \quad \text{for} \quad x \in [0, 1/(8k)] \cup [1 - 1/(8k), 1].
    \end{AlignedEquation}
    Consequently, $\hat \FL_k(x)\geqslant 1$ for $x\in [0, 1/(8N)] \cup [1 - 1/(8N), 1]$. This allows us to derive the lower bound  
    \begin{AlignedEquation}
        \intT e^{\sum_{k=1}^N \frac{\sqrt 2 \kappa}{1+D\EL_k}\hat \FL_k}dx \geqslant 2\int_0^{1/8N} e^{\sum_{k=1}^N \frac{\sqrt 2 \kappa}{1+D\EL_k}\hat \FL_k}dx \geqslant e^{\sum_{k=1}^N \frac{\sqrt 2 \kappa}{1+D\EL_k}} \frac{1}{4N}
    \end{AlignedEquation}
    which yields
    \begin{AlignedEquation}
        \label{eq:FunRoughEstim}
        \kappa \log \left(\intT e^{\sum_{k=1}^N \frac{\sqrt 2\kappa}{1+D\EL_k}\hat \FL_k}dx\right) \geqslant  \kappa^2\sum_{k=1}^N \frac{\sqrt 2}{1+D\EL_k} - \kappa\log(4N).
    \end{AlignedEquation}
    Applying inequality \eqref{eq:FunRoughEstim} to equation \eqref{eq:FunPoint}, we obtain
    \begin{AlignedEquation}
        \label{eq:UpperFunctionalEstimate}
        \mathcal{J}\left(\varphi_N\right) \leqslant -\frac{1}{2}\kappa^2 -\kappa^2\sum_{k=1}^N \frac{\sqrt 2 - 1}{1+D\EL_k} + \kappa\log(4N) \leqslant -\frac{1}{2}\kappa^2 -\kappa^2(\sqrt 2 - 1)\frac{N}{1+D\EL_N} + \kappa\log(4N).
    \end{AlignedEquation}
    To proceed with our analysis, we examine the conditions ensuring that
    \begin{AlignedEquation}
        \label{eq:Neg}
        -\frac{1}{2}\kappa^2-\kappa^2(\sqrt 2 - 1)\frac{N}{1+D\EL_N} + \kappa\log(4N) < -\frac{1}{2}\kappa^2
    \end{AlignedEquation}
    holds true for some $N\in\N$. Through algebraic manipulation, inequality \eqref{eq:Neg} can be equivalently expressed as
    \begin{AlignedEquation}
        \label{eq:DEstim}
        D < \frac{\kappa(\sqrt 2 - 1)N - \log(4N)}{\EL_N\log(4N)}
    \end{AlignedEquation}
    We observe that the right-hand side of inequality \eqref{eq:DEstim} converges to zero as $N\to\infty$ and is positive for sufficiently large $N$, thus ensuring the existence of a global maximum. Accordingly, we define
    \begin{AlignedEquation}
        \label{eq:MaxN}
        D_1(\kappa) := \max_{N\in\N}\frac{\kappa N(\sqrt2-1) - \log(4N)}{\EL_N\log(4N)}.
    \end{AlignedEquation}
    Consequently, for any $D < D_1(\kappa)$, we obtain the strict inequality $\mathcal{J}(\varphi_N) < -\kappa^2/2$.

    An alternative approach to establish that the infimum is smaller than $\mathcal{J}(\overline U)$ involves analyzing the convexity properties of the functional $\mathcal{J}$ at the constant solution.
    %the negativity of the infimum
    From formula \eqref{eq:HesSecDer}, the second derivative at $\overline U\equiv \kappa$ is expressed as $J_{\FL_1\FL_1}(\overline U) = 1+D\EL_1 - \kappa$. When $\kappa >1$ and $D<D_2(\kappa):=(\kappa-1)/\EL_1$, the functional exhibits strict concavity in the direction of $\FL_1$, leading again to $\inf \mathcal{J}(\varphi) < -\kappa^2/2$.
    We subsequently define
    \begin{AlignedEquation}
    D_{min}(\kappa) = \begin{cases}
    D_1(\kappa), & \kappa \in (0,1], \\
    \max \big(D_1(\kappa), D_2(\kappa)\big), & \kappa >1,
    \end{cases}
    \end{AlignedEquation}
    which ensures that $\inf \mathcal{J} < \mathcal{J}(\kappa)$ whenever $D < D_{min}$.
\end{proof}

\begin{remark}
    In Lemma \ref{thm:ExistenceOfMinus}, we introduced two critical diffusion values, $D_1$ and $D_2$, to characterize distinct behavioral regimes of the functional $\mathcal{J}$. For $D>D_2$ (or for any $D$ when $\kappa\leqslant1$), the constant function $\overline U(x) \equiv \kappa$ is a local minimum of $\mathcal{J}$. However, in the regime where $D_1>D_2$ (which occurs for sufficiently small values of $\kappa$), the functional may admit multiple local minima, corresponding to coexisting stable steady states (see Remark \ref{rem:Stab}).
\end{remark}

Next we deal with the case when the diffusion is large.

\begin{lemma}
    \label{thm:ZeroOnly}
    For each $\kappa >0$ there exist a constant $D_{max}(\kappa)$ such that if $D>D_{max}(\kappa)$ then the functional $\mathcal J$ is globally convex.
\end{lemma}

\begin{proof}
    It is sufficient to prove that the second derivative $D^2\mathcal J (\varphi) (\psi, \psi)$ given by the formula \eqref{eq:Ders} is nonnegative 
    for all $\varphi, \psi \in W^{1,2}(\T)$ (see \textit{e.g.} \cite[Corollary 3.8.6]{constantin2019convex}).
    First, fix functions $\varphi,\psi\in W^{1,2}(\T)$ and set 
    \begin{AlignedEquation}
        \psi = \sum_{k=0}^\infty R_k \FL_k(x) \quad \text{and} \quad \psi_n = \sum_{k=0}^n R_k \FL_k(x) \quad \text{with} \quad R = (R_0, \ldots) \in \omega_\EL^2.      
    \end{AlignedEquation} 
    Since $\{\FL_k\}_{k=0}^\infty$ is an orthonormal basis of $L^2(\T)$ and an orthogonal basis of $W^{1,2}(\T)$, we have $\psi_n \to \psi$ in $L^2(\T)$ and $W^{1,2}(\T)$.
    In particular, by embedding $W^{1,2}(\T)\subseteq L^\infty(\T)$ there holds
    \begin{AlignedEquation}
        \sup_{n\in\N_0} \|\psi_n\|\Li \leqslant C\sup_{n\in\N_0} \|\psi_n\|_{W^{1,2}} \leqslant C \|\psi\|_{W^{1,2}} < \infty
    \end{AlignedEquation} 
    which implies
    \begin{AlignedEquation}
        \left|\intT e^\varphi (\psi_n^2 - \psi^2) dx\right| \leqslant e^{\|\varphi\|\Li} \left( \|\psi_n\|\Lt^2 - \|\psi\|\Lt^2 \right) \leqslant C \|\psi_n - \psi\|\Lt \|\psi_n + \psi\|\Lt \leqslant C \|\psi_n - \psi\|\Lt.
    \end{AlignedEquation}
    Similarly using $\|\psi\|\Lo\leqslant\|\psi\|\Lt$ we obtain
    \begin{AlignedEquation}
        \left| \left( \intT e^\varphi \psi_n dx \right)^2 - \left( \intT e^\varphi \psi dx \right)^2 \right| \leqslant e^{2\|\varphi\|\Li} \left( \|\psi_n\|\Lo^2 - \|\psi\|\Lo^2 \right) \leqslant C \|\psi_n - \psi\|\Lt.
    \end{AlignedEquation}
    Now since $\psi_n \to \psi$ in $W^{1,2}(\T)$ passing to the limit yields $D^2\mathcal{J}(\varphi)(\psi_n,\psi_n) \to D^2\mathcal{J}(\varphi)(\psi,\psi)$ and hence it is enough to prove $D^2\mathcal J(\varphi)(\psi_n,\psi_n) \ge 0$ for all $n\in\N_0$.

    For each $n\in\N$, the derivative defines a symmetric quadratic form $D^2\mathcal J(\phi)(\psi_n,\psi_n) = \langle H^nR, R\rangle$ for $R\in \R^{n+1}$. The coefficients of the associated matrix $H^n(\varphi) = (\mathcal{J}_{\FL_i\FL_j}(\varphi))_{i,j=0}^n\in \R^{(n+1)\times (n+1)}$ are given by formulas \eqref{eq:SecDer0} and \eqref{eq:HesSecDer} and by inequalities \eqref{eq:CoefEstim} they satisfy $h_{ii} \geqslant 1+D\EL_i - 2\kappa$ and $|h_{ij}| \leqslant 4\kappa$ for all $i,j\in\N$.
    Since $H^n(\varphi)$ is symmetric, Sylvester's criterion dictates that it suffices to show
    \begin{AlignedEquation}
        \det \left( H^n(\varphi) \right) > 0 \quad \text{for each} \quad n\in \N_0
    \end{AlignedEquation}
    in order to obtain that $H^n(\varphi)$ is positive definite for all $n\in\N_0$.
    As an immediate consequence of \eqref{eq:CoefEstim} we have $\det(H^0(\varphi)) = 1$ and $\det (H^n(\varphi)) = \det (\mathcal{J}_{\FL_i, \FL_j}(\varphi))_{i,j=1}^n$. 
    In the following we assume that diffusion is sufficiently large, namely $1+D\EL_1 - 2\kappa>0$, implying $h_{ii}>0$ for all $i\in\N$.
    Let us define a family of subsets $\mathcal{A}_k = \{A \subset \{1, \ldots, n\} : |A| = k\}$. Using the Leibniz formula for determinants and the estimates for the coefficients of $H^n(\varphi)$, we can write
    \begin{AlignedEquation}
        \label{eq:SylvDet}
        \det \left( H^n(\varphi) \right) = \sum_{\sigma \in S^n} \left( {\rm sgn}(\sigma) \prod_{i = 1}^n h^n_{i,\sigma(i)} \right) \geqslant \prod_{i=1}^n h_{ii} - \sum_{k = 1}^{n} \left(4\kappa \right)^k k! \left(\sum_{A \in \mathcal{A}_{n-k}} \prod\limits_{j\in A} h_{jj}\right).
    \end{AlignedEquation}
    To obtain the estimate, we count with $k$ the number of off-diagonal elements in $\prod_{i = 1}^n h^n_{i,\sigma(i)}$ for some $\sigma\in S^n$ and group permutations with the same number of off-diagonals together.
    The first term $\prod_{i=1}^n h_{ii}$ corresponds to $\sigma = id$ and thus $k=0$.
    The product of $k$ off-diagonal elements is estimated by $(4\kappa)^k$ and multiplied by $k!$ as there are $k!$ possibilities to choose the order of these $k$ off-diagonal elements.
    The remaining term $\sum_{A \in \mathcal{A}_{n-k}} \prod\limits_{j\in A} h_{jj}$ corresponds to the sum of all possible choices of building the product of $n-k$ diagonal elements. Dividing equation \eqref{eq:SylvDet} by the product of all diagonal elements we obtain that the determinants are positive if
    \begin{AlignedEquation}
        1 > \sum_{k = 1}^{n} (4\kappa )^k k! \left(\sum_{A \in \mathcal{A}_{k}} \prod\limits_{j\in A} \frac{1}{h_{jj}}\right) \qtext{for all} n\in\N.
    \end{AlignedEquation}
    Since by inequalities \eqref{eq:CoefEstim}, the diagonal elements are bounded from below, we obtain that for sufficiently large $D$, the series $\sum (h_{ii})^{-1}$ is bounded
    \begin{AlignedEquation}
        \sum_{i=1}^n \frac{1}{h_{ii}} \le \sum_{k=1}^\infty\frac{1}{(1 + D\EL_k - 2\kappa)} = C(D) \xrightarrow{D\to\infty} 0. 
    \end{AlignedEquation}
    Hence, it holds
    \begin{AlignedEquation}
        \sum_{A \in \mathcal{A}_{k}} \prod\limits_{j\in A} \frac{k!}{h_{jj}} \leqslant \left(\sum_{i=1}^n \frac{1}{h_{ii}}\right)^k \le C(D)^k.
    \end{AlignedEquation}
    Choosing $D_{max}>(2\kappa-1)/\EL_1$ sufficiently large such that $C(D_{max}) < (8\kappa )^{-1}$,  we obtain
    \begin{AlignedEquation}
        \sum_{k = 1}^{n} (4\kappa)^k\cdot  \left(\sum_{A \in \mathcal{A}_{k}} \prod\limits_{j\in A} \frac{k!}{h_{jj}}\right) \leqslant \sum_{k = 1}^{n} \left(\frac{4\kappa}{8\kappa }\right)^k < 1 \qquad \text{for all} \qquad D > D_{\max}.
    \end{AlignedEquation}
    Thus, for $D > D_{\max}$, the matrix $H^n(\varphi)$ is positive definite for all $n\in\N_0$ and for all $\varphi\in W^{1,2}(\T)$.
    It follows $D^2\mathcal{J}(\varphi)(\psi,\psi)\geqslant 0$ for all $\varphi,\psi\in W^{1,2}(\T)$ for $D>D_{max}$ and the functional $\mathcal{J}$ is globally convex.
\end{proof}

\begin{remark}
    Assume $D > a \kappa$ where $a>0$ is a constant to be determined later on. Then we have 
    \begin{AlignedEquation}
        \sum_{k=1}^\infty \frac{1}{1+D\EL_k - 2\kappa} \leqslant \sum_{k=1}^\infty \frac{1}{1+\kappa(a\EL_k - 2)} \leqslant \frac{1}{\kappa}\sum_{k=1}^\infty \frac{1}{(a\EL_k - 2)}.
    \end{AlignedEquation}
    Here we assume $a >2/\EL_1$ to ensure that all terms are positive and the series converges. By the direct computation we obtain that
    \begin{AlignedEquation}
        \sum_{k=1}^\infty \frac{1}{(a\EL_k - 2)} = \frac{1}{4}\left(1-\frac{\sqrt{2} \pi \cot \left(\frac{\sqrt{2} \pi}{\sqrt{a}}\right)}{\sqrt{a}}\right)
    \end{AlignedEquation}
    and the condition
    $\sum_{k=1}^\infty (a\EL_k - 2)^{-1} \leqslant {1}/{8}$ is satisfied for $a>15/\EL_1$.
    This means, we could choose $D_{max}(\kappa) = 15 \kappa / \EL_1$ in Lemma \ref{thm:ZeroOnly}.
\end{remark}

\subsection{Existence of spatial patterns}
We finish this section by proving the existence of  nonconstant solutions to problem \eqref{equ_varphi_Stat} by incorporating the geometry of the functional $\mathcal{J}$. It is clear that each critical point of the functional $\mathcal{J}$ corresponds to a weak solution of problem \eqref{equ_varphi_Stat}.

\begin{proof}[Proof of Theorem \ref{thm:ExistenceStationary}]
    First we show the existence of nonconstant stationary solutions. Combining Lemma \ref{thm:Coercive} and Lemma \ref{thm:WeaklyLower} with a classical extreme value theorem, \cite[Thm 6.2.8]{drabek2013methods}, we conclude that there exists a minimizer $U\in W^{1,2}(\T)$ satisfying $\inf \mathcal{J} = \mathcal{J}(U)$ and the weak formulation of \eqref{equ_varphi_Stat}. By Lemma \ref{thm:ExistenceOfMinus} we immediately obtain $\mathcal{J}(U) < \mathcal{J}(\overline U)$ for $D<D_{min}$ and since there exists only one constant critical point of $\mathcal J$ (only one constant solution to equation \eqref{equ_varphi_Stat}) we immediately conclude that $U(x)$ is nonconstant.   
    {As $U$ is a local minimum of $\mathcal{J}$, the steady state $U$ is Lyapunov stable in $L^2(\T)$.}
    
    Next we show the uniqueness of the stationary solutions for $D>D_{max}$. As by Lemma \ref{thm:ZeroOnly} the functional $\mathcal{J}$ is globally convex, every local minimum is already a global minimum.
    Furthermore, the set of all minima forms a convex set and is, in particular, connected (see \textit{e.g.} \cite[Theorem 3.10.1]{constantin2019convex}).
    Since $\overline U \equiv \kappa$ is a critical point of $\mathcal J$, formula \eqref{eq:Ders} together with Poincare inequality $\EL_1\intT (\psi - \int\psi)^2 \leqslant\int\psi_x^2$ yields
    \begin{AlignedEquation}
        D^2\mathcal J(\overline{U})(\psi, \psi) 
        &= D\intT \psi_x^2 + \intT \psi^2 -\kappa\intT\psi^2 + \kappa \left(\intT\psi\right)^2 \\
        &\geqslant 
        (1 + D\EL_1 -\kappa)\intT \psi^2  
        + (\kappa - D\EL_1) \left(\intT \psi\right)^2 \\
        &\geqslant \intT \psi^2 + ( D\EL_1 - \kappa)\left(\int\psi^2 - \left(\int\psi\right)^2\right).
    \end{AlignedEquation}
    Hence since $0\leqslant \left(\int\psi^2 - \left(\int\psi\right)^2\right) \leqslant \int\psi^2$ we have 
    \begin{AlignedEquation}
        D^2\mathcal J(\overline{U})(\psi, \psi) \geqslant \begin{cases}
            \int\psi^2 & \text{for } (D\EL_1 - \kappa) >0, \\
            (1+D\EL_1-\kappa)\int\psi^2 & \text{for } (D\EL_1 - \kappa) <0.
        \end{cases}
    \end{AlignedEquation}
    Since we consider $D>D_{max}$, 
    % where $D_{max}$ is chosen such that $C(D_{max})<(8\kappa)^{-1}$, 
    it is clear that $1 + D\EL_1 -\kappa > 0$.
    We will later see in Proposition \ref{thm:Stab_constant_steady_state} that the constant steady state $\overline{U}$ is stable exactly when $1 + D\EL_1 -\kappa > 0$ holds. Thus, this condition is really necessary to prove convexity of $\mathcal{J}$. By the strict positivity of $D^2\mathcal{J}(\overline U)(\psi, \psi)$, it follows that $\mathcal{J}$ is strictly convex in a small neighborhood of $\overline U$ and in this neighborhood, $\overline{U}$ is the unique minimum of $\mathcal{J}$.
    Since in case of multiple minima, they would be connected, $\overline U$ is the unique critical point of $\mathcal{J}$ and hence a unique solution to equation~\eqref{equ_varphi_Stat}. {Moreover, as the unique critical point and global minimum of $\mathcal{J}$, the constant steady state $\overline U$ is globally asymptotically stable in $L^2(\T)$.}
\end{proof}

\section{Linear stability analysis}\label{sec:Stability}

In this section, we examine the linear stability of stationary solutions to problem \eqref{equ_varphi}, which allows us to identify dynamically admissible patterns. 
Let $U$ be a solution of equation \eqref{equ_varphi_Stat}. To determine the linear stability of $U$, we insert $u(x,t) = U(x) + e^{\ENS t} \varphi(x)$ into \eqref{equ_varphi} and linearize about $U$, which leads to the eigenvalue problem
\begin{AlignedEquation}\label{eq:Nonlocal_eigv_problem}
    D \varphi_{xx} + \left(\kappa\frac{e^U}{\intT e^U \dy} - 1 - \ENS\right) \varphi = \kappa \frac{e^U}{(\intT e^U \dy)^2} \intT e^U \varphi \, \dy
\end{AlignedEquation}
supplemented with the boundary conditions $\varphi(0) = \varphi(1)$ and $\varphi_x(0) = \varphi_x(1)$.
Throughout our analysis, we employ the notation $\ENS$ to denote the eigenvalues of this nonlocal problem. Later, we introduce $\ELS$ to represent the eigenvalues of the corresponding local problem obtained by setting the right-hand side of \eqref{eq:Nonlocal_eigv_problem} to zero.

\subsection{Constant steady state}\label{sec:Stability_constant}

First we consider stability of the constant steady state.
% By Theorem~\ref{thm:ExistenceStationary}, we know that
Clearly, $\overline U \equiv \kappa$ is the unique constant steady state of problem \eqref{equ_varphi_Stat} which exists for all parameters $D$, $\kappa$.
To investigate its stability, we insert $\overline U \equiv \kappa$ in \eqref{eq:Nonlocal_eigv_problem} and obtain the problem
\begin{AlignedEquation}
    \label{eq:EigCst}
    D \varphi_{xx} + \left(\kappa - 1 - \ENS\right) \varphi = \kappa  \intT \varphi \, d y.
\end{AlignedEquation}

\begin{proposition}\label{thm:Stab_constant_steady_state}
    The constant solution $\overline U\equiv\kappa$ is stable for $\kappa < 1+\EL_1 D$ and unstable for $\kappa > 1+\EL_1 D$.
\end{proposition}
\begin{proof}
The spectrum of \eqref{eq:EigCst} is fully characterized by the 
set of eigenfunctions of the Laplace operator (see Proposition \ref{prop:Spectral}). Choosing  $\varphi\in\{\FL_k\}_{k=0}^\infty$ and using the fact that $\intT \FL_k(x) dx = 0$ for $k\geqslant 1$, we find that the  eigenvalues of the linearized nonlocal operator \eqref{eq:EigCst} are
\begin{AlignedEquation}
    \ENS_k = \begin{cases}
        -1& \text{for} \ k = 0, \\
        \kappa - 1 - D\EL_k & \text{for} \ k \geqslant 1.
    \end{cases}
\end{AlignedEquation}
Notice, that since 
% $\ENS_0 < 0$ and 
$\ENS_1 > \ENS_k$ for all $k\ge2$, the claim follows.
\end{proof}

\subsection{Local and nonlocal Sturm-Liouville}
We now turn our attention to nonconstant steady states $U$.
By Proposition \ref{prop:ExistenceGlobal}, we know that $U\in W^{2,2}(\T)$.
To analyze the spectrum, we compare the eigenvalues of the linearized operator with those arising from the associated local periodic Sturm–Liouville problem. Our approach builds upon the methods developed in \cite{bose1998stability, coddington1956theory}, which we adapt to the setting of periodic boundary conditions.

In equation \eqref{eq:Nonlocal_eigv_problem}, introducing the notation
\begin{AlignedEquation}\label{eq:definition_A_C_M}
    A(x) = \kappa\frac{e^{U(x)}}{\intT e^{U(y)} d y} - 1,\quad M = \frac{\kappa}{\left(\intT e^{U(y)} d y\right)^2}>0,\quad C(x) = e^{U(x)}
\end{AlignedEquation}
transforms the equation into the eigenvalue problem
\begin{AlignedEquation}\label{eq:Nonlocal_problem}
    \begin{gathered}
        D \varphi_{xx} + (A(x) - \ENS) \varphi = M C(x) \intT C(y) \varphi(y) d y,
    \end{gathered}
\end{AlignedEquation}
supplemented with periodic boundary conditions $\varphi(0) = \varphi(1)$ and $\varphi_x(0) = \varphi_x(1)$.
We consider the linearization $\mathcal{L}$ around $U$ as operator on $L^2(\T)$, i.e.\
\begin{align*}
    \mathcal{L}: \mathcal{D}(\mathcal{L}) = W^{2,2}(\T) \subset L^2(\T) \to L^2(\T),\quad \mathcal{L}\varphi = D \varphi_{xx} + A(x) \varphi - M C(x) \intT C(y) \varphi(y) d y.
\end{align*}
Note that the functions $A$ and $C$ are bounded in $x\in\T$ as $U\in W^{2,2}(\T) \subset L^\infty(\Omega)$.
The stability of $U$ is then determined by the sign of the real part of the eigenvalues by Lemma \ref{lem:PropSpecL}: if $\Re\ENS < 0$ for all eigenpairs $(\ENS, \phi)$ of \eqref{eq:Nonlocal_problem}, then $U$ is linearly stable; conversely, if there exists an eigenvalue $\ENS$ with $\Re\ENS > 0$, then $U$ is linearly unstable.
\begin{lemma}\label{lem:PropSpecL}
    The spectrum of $\mathcal{L}$ is discrete, i.e. consists only of isolated eigenvalues of the eigenvalue problem \eqref{eq:Nonlocal_problem} with finite multiplicities. The eigenvectors of $\mathcal{L}$ are an orthogonal basis of $L^2(\T)$.
    Moreover, the all eigenvalues are real and countable.
\end{lemma}
\begin{proof}
    The operator $\mathcal{L}$ is self-adjoint with respect to the $L^2$-scalar product which can be seen using the specific form of the right-hand-side.
    Thus, all eigenvalues are real.
    Moreover, $\mathcal{L}$ has a compact resolvent for $\lambda \in\rho(\mathcal{L})$.
    This is proven by constructing the resolvent $(\lambda I - \mathcal{L})^{-1}$ using the Lax-Milgram Theorem and proving compactness using the compact embedding $W^{1,2}(\T) \hookrightarrow L^2(\T)$.
    Hence, the spectrum of $\mathcal{L}$ consists entirely of isolated eigenvalues.
    As $L^2(\T)$ is a separable Hilbert space, the spectral theorem for compact self-adjoint operators states that a compact, self-adjoint operator has a basis of eigenvectors.
    Since $\mathcal{L}$ and the compact operator $(\lambda I - \mathcal{L})^{-1}$ share the same eigenvectors, the eigenvectors of $\mathcal{L}$ are a orthogonal basis of $L^2(\T)$.
\end{proof}
Problem \eqref{eq:Nonlocal_problem} is closely related to the periodic Sturm–Liouville problem, apart from the presence of the nonlocal term on the right-hand side. Indeed, the eigenvalues $\ENS$ of \eqref{eq:Nonlocal_problem} can be related to solutions of the associated Sturm–Liouville problem
\begin{AlignedEquation}\label{eq:Local_problem}
\begin{gathered}
    D \psi_{xx} + (A(x) - \ELS) \psi = 0,
\end{gathered}
\end{AlignedEquation}
together with the boundary conditions $\psi(0) = \psi(1)$ and $\psi_x(0) = \psi_x(1)$. 
To distinguish the eigenpairs of both problems 
% \eqref{eq:Nonlocal_problem} and \eqref{eq:Local_problem} 
we use $(\ENS,\FNS)$ for solutions of \eqref{eq:Nonlocal_problem} and $(\ELS, \FLS)$ for solutions of \eqref{eq:Local_problem}.

The local periodic Sturm-Liouville problem was extensively studied in the past (see. \textit{e.g.} {\cite{constanda2010,lebovitz}}). Here we recall the well-known result for completeness of the exposition.
\begin{proposition}
    \label{thm:Eigenvalues_local_problem}
The solutions of \eqref{eq:Local_problem} are given by pairs $\{(\ELS_n, \FLS_n)\}_{n\in\N_0}$, where the set $\{\FLS_n\}_{n\in\N_0}$ is an orthonormal basis of $L^2(0,1)$.
The eigenvalues satisfy
\begin{AlignedEquation}
    \ELS_{0} > \ELS_{1} \geqslant \ELS_{2} > \ldots > \ELS_{2j+1} \geqslant \ELS_{2j+2} > \ELS_{2(j+1)+1} \geqslant \ELS_{2(j+1)+2} > \ldots.
\end{AlignedEquation}
% for $j = 0, 1, \dots$.
The corresponding eigenfunctions $\FLS_{2j+1}$, $\FLS_{2j+2}$ have exactly $2j+2$ zeros in the interval $[0,1)$. In particular $\FLS_0$ does not vanish on $[0,1]$ and the eigenvalue $\ELS_0$ is simple. 
\end{proposition}

If an eigenfunction $\FLS_n$ of the local problem \eqref{eq:Local_problem} satisfies $\int_{0}^{1} C(y) \FLS_n(y)d y = 0$, then the pair $(\ELS_n, \FLS_n)$ also solves the nonlocal problem \eqref{eq:Nonlocal_problem} and $\ELS_n$ is a nonlocal eigenvalue.
Since the ordering of the eigenvalues $\ELS_n$ is well understood, we first focus on characterizing the eigenvalues of the nonlocal problem that do not coincide with those of the local problem.
\begin{lemma}\label{lem:equation_for_lambda}
% Assume that there is at least one local eigenvalue $\ELS_n$ which is not a nonlocal eigenvalue.
If a nonlocal eigenvalue $\ENS$ of \eqref{eq:Nonlocal_problem} satisfies $\ENS\notin \lbrace\ELS_n\rbrace_{n=0}^\infty$, then $\ENS$ solves
\begin{AlignedEquation}\label{eq:equation_for_lambda}
    \frac{1}{M} = \sum_{n\in\N_0} \frac{\beta_n^2}{\ENS - \ELS_n},
\end{AlignedEquation}
where $\beta_n = \int_{0}^{1} C(y) \FLS_n(y)d y$. 
Conversely if $\ENS\in\R$ solves \eqref{eq:equation_for_lambda}, then $\ENS$ is an eigenvalue of the nonlocal problem \eqref{eq:Nonlocal_problem}.
\end{lemma}
\begin{proof}
Let $(\ENS, \FNS)$ be an eigenpair of the nonlocal problem \eqref{eq:Nonlocal_problem}. Assume that $\ENS$ is not a solution of the local problem \eqref{eq:Local_problem}, i.e.
\begin{AlignedEquation}
    J := \intT C(x) \FNS(x) d x \ne 0.
\end{AlignedEquation}
We decompose nonlocal eigenfunction $\FNS$ in terms of the orthonormal eigenbasis $\{\FLS_n\}_{n\in\N_0}$ of the local problem $\FNS(x) = \sum a_n \FLS_n$ with $(a_n)_{n\in\N_0} \in \ell^2(\N_0)$. 
Multiplying equation \eqref{eq:Nonlocal_problem} by $\FLS_m$ and integrating yields
\begin{AlignedEquation}
    &\intT \sum_{n\in\N_0} a_n \FLS_n \big( D  (\FLS_m)_{xx} + (A(x)-\ENS) \FLS_m\big) d x = JM\beta_m, %\\
    %&\qquad\qquad= \intT B(x) \intT C(y) \varphi(y) d y\ \psi_m(x) d x = Jk\beta_m,
\end{AlignedEquation}
where $\beta_n = \intT C(x) \FLS_n(x) d x$ and we have applied integration by parts twice.
Next, adding and substracting $\ELS_m \FLS_m$ and using  orthogonality of the basis $\FLS_n$ together with equation \eqref{eq:Local_problem} for $(\ELS_m, \FLS_m)$, we obtain $a_m(\ELS_m-\ENS) = JM\beta_m$, which is equivalent to
\begin{AlignedEquation}
    a_m = J \frac{M\beta_m}{\ELS_m-\ENS}.
\end{AlignedEquation}
Therefore we can express $J$ in the form
\begin{AlignedEquation}
    J &= \intT C(x) \FNS(x) d x = \sum_{n\in\N_0} a_n \beta_n = J M \sum_{n\in\N_0} \frac{\beta_n^2}{\ELS_n-\ENS}.
\end{AlignedEquation}
Since $J\ne 0$ by assumption, $\ENS$ solves \eqref{eq:equation_for_lambda}.
The series converges for each $\ENS\ne \ELS_n$ as $\ELS_n \sim n^2$ \cite{coddington1956theory, lebovitz} and $\beta_n$ is uniformly bounded.
This proves one direction of the claim.

Now, we fix $\ENS\in\R$ which solves \eqref{eq:equation_for_lambda}. Notice that $\ENS \ne \ELS_n$ for all $n\in\N_0$ and there exist $\beta_n \ne 0$.
% By assumption, we know that at least one such $\beta_n$ exists.
We set $a_n = {M \beta_n}/({\ELS_n - \ENS})$ 
and define the function $\FNS = \sum_{n\in\N_0} a_n \FLS_n$ which is nonzero as $\beta_n\ne0$ for at least one $n\in\N_0$.
Substituting $\varphi$ into the equation \eqref{eq:Nonlocal_problem} and again adding and substracting $\ELS_n\FLS_n$ and utilizing \eqref{eq:Local_problem} we obtain
\begin{AlignedEquation}
    M\sum_{n=0}^\infty \beta_n \FLS_n = MC(x) \sum_{n=0}^\infty a_n \intT C(x) \FLS_n.
\end{AlignedEquation}
Notice that by assumption \eqref{eq:equation_for_lambda} right hand side satisfies
\begin{AlignedEquation}
    \sum_{n=0}^\infty a_n \intT C(x) \FLS_n = \sum_{n=0}^\infty a_n\beta_n =  \sum_{n=0}^\infty M\frac{\beta_n^2}{\nu-\lambda} = 1
\end{AlignedEquation}
and since $\sum \beta_n \FLS_n = C(x)$ the proof is completed.
\end{proof}

\subsection{Relation between eigenvalues}
Next, we establish the location of eigenvalues $\ENS$ of the nonlocal problem \eqref{eq:equation_for_lambda}.
First we decompose the set of local eigenvalues into two subsets 
\begin{AlignedEquation}
    \lbrace \ELS_n \rbrace_{n=0}^\infty = \lbrace \ELS_n: \beta_n  = 0 \rbrace \cup \lbrace \ELS_n: \beta_n  \neq 0 \rbrace:= \Lambda_0 \cup \Lambda_\beta 
\end{AlignedEquation}
We must account for the terms where $\beta_n = 0$ as they vanish in the series \eqref{eq:equation_for_lambda}. These sets are not necessarily disjoint, \textit{i.e.} there can exists a double eigenvalue $\lambda_{n} = \lambda_{n+1}$ and $\lambda_{n}\in\Lambda_0$ while $\lambda_{n+1}\in\Lambda_\beta$. Without loss of generality we assume that the elements of $\Lambda_0, \Lambda_\beta$ are in monotonically decreasing order
\begin{AlignedEquation}
    \Lambda_0 = \lbrace \lambda_{n_1}, \lambda_{n_2}, \ldots \rbrace, \qquad \Lambda_\beta = \lbrace \lambda_{m_1}, \lambda_{m_2}, \ldots \rbrace
\end{AlignedEquation}
Notice that $\Lambda_0, \Lambda_\beta$ can be infinite, finite or empty and the eigenvalues can be double.
% \begin{remark}
% If we consider a symmetric steady state solution $U(x) = U(1-x)$
% % such as a multimodal solution with peaks symmetric about the mid point $x=1/2$, 
% then the functions $A(x)$ and $C(x)$ and all eigenfunctions $\FLS_n$ are also symmetric about $x=1/2$. One can prove in case of a double eigenvalue $\ELS_n = \ELS_{n+1}$, that at least one of the corresponding eigenfunctions, say $\FLS_n$, has an odd number of zeros in the interval $[0,1]$ and thus is an odd function \cite[Proof of Theorem 11.5.1]{lebovitz}. Since $C(x)>0$, the function $C(x)$ is even
% and $\beta_n = 0$ implying $\ELS_n \in \Lambda_0$.
% \end{remark}
% In case of a double eigenvalue $\ELS_n = \ELS_{n+1}$ satisfying $\beta_n\ne0$, $\beta_{n+1}\ne0$, i.e.\ $\ELS_n, \ELS_{n+1} \in \{\ELSa_n\}_{n=0,\dots,\Na}$, we delete one copy and set the corresponding $\beta_n$ to $\sqrt{\beta_n^2 + \beta_{n+1}^2}$. This does not change the series in equation \eqref{eq:equation_for_lambda}.
% Then, the equation \eqref{eq:equation_for_lambda} transforms to 
% \begin{AlignedEquation}\label{eq:equ_lambda_transformed}
% \frac{1}{M} = \sum_{n=0}^{\Na} \frac{\beta_n^2}{\ELSa_n - \ENS},
% \end{AlignedEquation}
% where $\ELSa_n$ is distinct for all $n=0,\dots,\Na$.
Using the introduced notation, equation \eqref{eq:equation_for_lambda} transforms to
\begin{AlignedEquation}\label{eq:equ_lambda_transformed}
\frac{1}{M} = \sum_{\ELS_{m_k}\in\Lambda_\beta} \frac{\beta_{m_k}^2}{\ELS_{m_k} - \ENS}.
\end{AlignedEquation}

With this preparation, we can characterize all eigenvalues of the nonlocal problem \eqref{eq:Nonlocal_problem}.
\begin{theorem}\label{thm:Charakt_lambda}
The eigenvalues of the nonlocal problem \eqref{eq:Nonlocal_problem} are given by
\begin{AlignedEquation}
    \{\ENS_{n} \}_{n=1}^\infty = \Lambda_0 \cup \{\ENS\in\R : \ENS \text{ solves equation } \eqref{eq:equ_lambda_transformed}\}.
\end{AlignedEquation}
% The solutions $\ENS$ of \eqref{eq:equ_lambda_transformed} are characterized as follows:\\
For all distinct and subsequent local eigenvalues  $\lambda_{m_i}, \lambda_{m_{i+1}} \in \Lambda_\beta$ there exists exactly one  
% nolocal eigenvalue $\ENS$ 
solution $\nu$ of equation \eqref{eq:equ_lambda_transformed} with 
$\nu\in (\lambda_{m_{i+1}}, \lambda_{m_{i}})$. In particular if $\Lambda_\beta$ is finite, there exists exactly one solution of equation \eqref{eq:equ_lambda_transformed} with
% nonlocal eigenvalue 
$\nu\in(-\infty,\min\Lambda_\beta)$.
% For all $m\in \{0, \dots, \Na-1\}$, there exists an eigenvalue $\ENS$ of \eqref{eq:Nonlocal_problem} with $\ELSa_{m+1} < \ENS < \ELSa_m$.\\
% If $\Na < \infty$, there exists an eigenvalue $\ENS$ of \eqref{eq:Nonlocal_problem} with $-\infty < \ENS < \ELSa_{\Na}$.\\
%For all $m\in\{0, \dots, \Nb\}$, there exists an eigenvalue $\ENS$ of \eqref{eq:Nonlocal_problem} with $\ENS = \ELSb_m$.
\end{theorem}
\begin{proof}
It is clear $\Lambda_0 \subseteq \lbrace \nu_n\rbrace_{n=1}^\infty$ and thus the first part follows from Lemma \ref{lem:equation_for_lambda}.\\
For the second part first note that all solutions $\ENS$ of \eqref{eq:equ_lambda_transformed} are real. They can be estimated as follows.
Fix subsequent local eigenvalues $\ELS_{m_i}, \ELS_{m_{i+1}} \in\Lambda_\beta$.
From now on we assume that $\ELS_{m_i}$ and $\ELS_{m_{i+1}}$ are distinct. If they coincide, we merge the two summands in the series \eqref{eq:equ_lambda_transformed} by setting $\beta_{m_i} = \sqrt{\beta_{m_i}^2 + \beta_{m_{i+1}}^2}$.

% Thus, by construction of $\Lambda_\beta$, we have $\beta_{m_i}\ne 0$, $\beta_{m_{i+1}}\ne0$ and $\ELS_{m_i} > \ELS_{m_{i+1}}$.

Then, the right-hand-side of \eqref{eq:equ_lambda_transformed} is continuous and monotone increasing in $\ENS \in(\ELS_{m_{i+1}}, \ELS_{m_i})$ and converging to $-\infty$ as $\ENS\downarrow\ELS_{m_{i+1}}$ and to $\infty$ as $\ENS\uparrow\ELS_{m_i}$.
Since $M>0$ by \eqref{eq:definition_A_C_M}, there exists a unique solution $\ENS_{m_i}$ of \eqref{eq:equ_lambda_transformed} in the interval $(\ELS_{m_{i+1}}, \ELS_{m_i})$.
Note that the right-hand-side of \eqref{eq:equ_lambda_transformed} is negative for all $\ENS > \ELS_{m_0}$ and, thus, there exists no solution $\ENS>\ELS_{m_0}$.
In case of finite $\Lambda_\beta$, there exists a solution $\ENS$ in the interval $(-\infty, \min\Lambda_\beta)$ as the right-hand-side of \eqref{eq:equ_lambda_transformed} converges to $0$ as $\ENS\downarrow -\infty$ and to $\infty$ as $\ENS\uparrow\min\Lambda_\beta$.
\end{proof}

\begin{remark}
    Applying this characterization to the eigenvalue problem for the constant steady solution $\overline U \equiv \kappa$, 
    % from Proposition \ref{thm:Stab_constant_steady_state},
    we see that $\Lambda_\beta = \lbrace\kappa-1\rbrace$ and $\Lambda_0 = \{\ELS_k\}_{k=1}^\infty$ and hence
    the nonlocal eigenvalues satisfy $\ENS_0 \in (-\infty, \kappa-1)$ and $\ENS_j = \ELS_j$ for all $j\geqslant 1$.
    Indeed, this coincides with the results from Proposition~\ref{thm:Stab_constant_steady_state}.
\end{remark}

\subsection{Stability of steady states}
With this characterization, we can find a criterion under which the steady state solution $U$ is (linearly) unstable.

\begin{theorem}\label{thm:Sufficient_conds_instability}
If the derivative $U_x$ of a stationary solution $U$ has at least three zeros in $[0,1)$, then $U$ is linearly unstable.
\end{theorem}
\begin{proof}
Since $U$ solves \eqref{equ_varphi_Stat}, we can differentiate \eqref{equ_varphi_Stat} with respect to $x$ and obtain that $U_x$ solves
\begin{AlignedEquation}
    0 = D (U_x)_{xx} - U_x + \kappa \frac{e^U}{\intT e^U d y} U_x = D(U_x)_{xx} + A(x) U_x,
\end{AlignedEquation}
which is exactly the local problem \eqref{eq:Local_problem} with $\ELS = 0$.
Since $U\in W^{2,2}(\T)$, all $U$, $U_x$ and $U_{xx}$ satisfy periodic boundary conditions.
Hence, the local problem has a zero eigenvalue with corresponding eigenfunction with at least $3$ zeros on the interval $[0,1)$.
Proposition \ref{thm:Eigenvalues_local_problem} implies that $\ELS_3 \geqslant 0 $ as the eigenfunctions  $\FLS_0$, $\FLS_1$, $\FLS_2$ have at most two zeros on the interval $[0,1)$. Therefore there are at least three strictly positive local eigenvalues $\ELS_0 > \ELS_1 \geqslant \ELS_2 > \ELS_3 \geq 0 $. This is already sufficient to obtain a positive nonlocal eigenvalue $\ENS$. Indeed, if all three local eigenvalues are in $\Lambda_\beta$ then by Theorem \ref{thm:Charakt_lambda} there exists a positive nonlocal eigenvalue $\ENS\in(\ELS_0,\ELS_1)$. Otherwise at least one of those local eigenvalues is also a nonlocal eigenvalue (belongs to $\Lambda_0$) and hence the proof is finished.
\end{proof}

\begin{remark}\label{rem:nu1=0_not_enough}
If $U$ is a steady state solution, with $U_x$ having two zeros in the interval $[0,1)$, then either $\ELS_1 = 0$ or $\ELS_2 = 0$.
Even though $\ELS_1=0$ implies $\ELS_0 > 0$, this is not enough to obtain instability of the nonlocal problem.
% Indeed, assume that $\ELS_0>0$ is not an eigenvalue of the nonlocal problem, namely \ $\ELS_0\in\Lambda_\beta$ and $\ELS_1 = 0$ is an eigenvalue of the nonlocal problem,  $\ELS_1\in\Lambda_0$.
Indeed, assuming $\ELS_0 \in \Lambda_\beta$ and $\ELS_1\in\Lambda_0$ implies $\Lambda_0\subseteq (-\infty,\,0].$ 
Moreover, by Theorem \ref{thm:Charakt_lambda} the nonlocal eigenvalues satisfy $\ENS_0 = 0$ and $\ENS_1 \in(\ELS_{m_1}, \ELS_0)$ with $\ELS_{m_1}\in\Lambda_\beta$.
If $\ELS_{m_1} <0$ it is possible that $\nu_1 <0$ and hence all nonlocal eigenvalues are nonpositive.
\end{remark}

To relate these results to the nonlinear dynamics, we exploit the fact that problem \eqref{equ_varphi} can be viewed as a gradient flow in $L^2(\T)$ with respect to the Lyapunov functional $\mathcal{J}$ defined in \eqref{eq:FunDef}.
This allows us to connect linear instability with nonlinear instability in the $L^2$-norm for perturbations in $W^{1,2}(\T)$.
\begin{lemma}\label{lem:linear_implies_nonlinear_instab} Let $U$ be a linear unstable steady state of \eqref{equ_varphi}. Then, the steady state is nonlinearly unstable with respect to the $L^2$-norm.
\end{lemma}
\begin{proof} The problem \eqref{equ_varphi} is a gradient flow with respect to the $L^2$-norm, thus is follows from the general theory of gradient flows and dynamical systems~\cite{ambrosio2005gradient} that nonlinearly stable solutions to problem \eqref{equ_varphi} are local minima of the Lyapunov functional $\mathcal{J}$ given by formula \eqref{eq:FunDef}. For the linearization $\mathcal{L}$ at $U$ we have \begin{align*} D^2\mathcal{J}(U)(\psi,\psi) = (-\mathcal{L}\psi, \psi)_{L^2}\quad \text{for all } \psi \in W^{2,2}(\T), \end{align*} compare to equation \eqref{eq:Ders}. Hence, the linearization corresponds to minus Hessian of $\mathcal{J}$ at $U$. By Lemma \ref{lem:PropSpecL}, the linear instability of $U$ implies the existence of a positive real eigenvalue $\lambda$. So if $\mathcal{L}$ has an eigenvalue $\lambda > 0$ with eigenvector $\varphi$, then \begin{align*} D^2\mathcal{J}(U)(\varphi, \varphi) = - \lambda \|\varphi\|_{L^2}^2 < 0. \end{align*} Hence, $U$ is not a local minimum of $\mathcal{J}$ and thus nonlinearly unstable. \end{proof}
\begin{remark}
In general, linear stability of a steady state does not necessarily imply nonlinear stability in the $L^2$-norm; it often only guarantees nonlinear stability in the $L^\infty$-norm \cite{KMMue2025}, or may require restricting the class of admissible perturbations \cite{Berlyand2026_nonlinStability} as we also did in the above lemma.
\end{remark}
With these results we can now prove our main result about the instability of $m$-modal solutions in Theorem \ref{thm:StabilityNCst}.

\begin{proof}[Proof of Theorem \ref{thm:StabilityNCst}]
    Let $U$ be an arbitrary stationary $m$-periodic solution to problem \eqref{equ_varphi}. Then, $U$ has at least $m$ peaks and hence, the derivative $U_x$ has at least $2m-1$ zeros in the interval $[0,1)$.
    By Theorem \ref{thm:Sufficient_conds_instability},  all $m$-modal solution $U$ are linearly unstable for $m\geqslant 2$.
By Lemma \ref{lem:linear_implies_nonlinear_instab}, nonlinear instability of all $m$-modal solutions follows.
    
    The periodic solution to problem \eqref{equ_varphi_Stat} of Theorem \ref{thm:ExistenceStationary} constructed in Section \ref{sec:Existence_steady_states} is a minimum of the Lyapunov functional given by formula \eqref{eq:FunDef}.
    % Recall, that problem \eqref{equ_varphi} is a gradient flow with respect to $L^2$-norm, thus is follows from the general theory of gradient flows and dynamical systems~\cite{ambrosio2005gradient} that this is a nonlinearly stable solution to problem \eqref{equ_varphi}.
    Therefore, the solution is nonlinearly stable and the solution is unimodal.
\end{proof}

\begin{corollary}[Bistability]
    \label{rem:Stab}
    For $\kappa < 1$ and sufficiently small $D>0$, the system exhibits bistability: both the constant state $U \equiv \kappa$ (stable by Proposition \ref{thm:Stab_constant_steady_state}) and the nonconstant unimodal solution (stable by Theorems \ref{thm:ExistenceStationary} and \ref{thm:StabilityNCst}) coexist.
\end{corollary}

\section{Bifurcation analysis}\label{sec:Bifurcation_analysis}
In this section we study nonconstant steady states to problem \eqref{equ_varphi}, bifurcating from the branch of constant solutions $\bar{\Phi}(\kappa) \equiv \kappa$.  Here, $\kappa$ is also a bifurcation parameter. 
The linearized operator $\mathcal{L}(\kappa)$ at a constant steady state $\bar{\Phi}(\kappa)$ is given by
\begin{align*}
    \mathcal{L}\left(\kappa\right)(\varphi) = D \varphi_{xx} + (\kappa - 1) \varphi - \kappa \intT \varphi(y) d y,
\end{align*}
and we emphasize its dependency on the bifurcation parameter $\kappa$.
Spectral properties of the linearized operator were discussed in Section \ref{sec:Stability_constant}. We recall, the eigenvalues of $\mathcal{L}\left(\kappa\right)$ are
\begin{align*}
    \nu_0 = -1, \quad \nu_k=\kappa - 1 - D\mu_k,
\end{align*}
where $\mu_k = 4 \pi^2 k^2$ are the eigenvalues of $-\Delta$ with periodic boundary conditions.
The eigenvalue $\nu_n$ crosses zero for each critical value $\kappa_n := 1 + D\mu_n$.
Thus, the linearized operator $\mathcal{L}(\kappa)$ has a zero eigenvalue for all $\kappa = \kappa_n$.
We denote the corresponding constant solution by $\bar{\Phi}_n \equiv \bar{\Phi}(\kappa_n)$. 

\subsection{Existence of branches}
To establish the existence of nonconstant solution branches bifurcating from $(\bar{\Phi}_n, \kappa_n)$, we apply the Crandall–Rabinowitz theorem for bifurcation from simple eigenvalues \cite[Theorem I.5.1]{Kielhoefer}.
% (Theorem~\ref{thm:BifSimplEigv}).
We first define the relevant function spaces:
\begin{AlignedEquation}
    Y &:= \left\{\varphi\in L^2(\T) \mid
         \varphi(x) = \varphi(1-x) \text{ for a.\ e.\ } x\in\T: \text{ symmetry constraint}
    \right\},\\
    X &:= \{\varphi\in W^{2,2}(\T) \mid \varphi(x) = \varphi(1-x)  \text{ for a.\ e.\ } x\in\T\}.
\end{AlignedEquation}
Note that the periodic boundary conditions of the Laplace operator $\Delta$ are naturally incorporated in the space $W^{2,2}(\T)$.
The symmetry condition ensures, the kernel of $\mathcal{L}(\kappa_n)$ is one-dimensional:
\begin{align*}
    \ker \mathcal{L}(\kappa_n) = \mathrm{span} \{\varphi_n\},\quad \text{where } \varphi_n := \sqrt{2} \cos(2n\pi x).
\end{align*}
We consider solutions in a sufficiently small neighborhood of $(\bar{\Phi}_n, \kappa_n)$ of the form
\begin{align*}
    \varphi = \bar{\Phi}(\kappa) + \chi, \quad \kappa = \kappa_n + \alpha, \quad \text{with } \chi \in X,\ \alpha\in\R.
\end{align*}
Define the mapping $G:X\times \R \to Y$ by
\begin{align*}
    G(\chi, \alpha) := D\chi_{xx} - \chi + (\kappa_n + \alpha) \left( \frac{e^\chi}{\int e^\chi d y} - 1 \right).
\end{align*}
Then $G(\chi, \alpha) = 0$ if and only if $\varphi = \bar{\Phi}(\kappa_n + \alpha) + \chi$ solves the stationary problem~\eqref{equ_varphi_Stat} for $\kappa = \kappa_n + \alpha$.
Note that, the trivial branch of constant steady states $\left(\bar{\Phi}(\kappa), \kappa\right)$ transforms to $(0, \alpha)$ and has a zero eigenvalue at $\alpha = 0$.

The map $G$ is well-defined as $X\subset L^\infty(\T)$ and thus $G(X\times\R)\subset L^2(\T)$.
To confirm the symmetry constraint, use that $\chi_x(x) = - \chi_x(1-x)$ if $\chi(x) = \chi(1-x)$ and $\chi$ regular enough. Hence, we have $\chi_{xx}(x) = \chi_{xx}(1-x)$ and $G(\chi,\alpha)(x) = G(\chi,\alpha)(1-x)$ follows.

We are now in a position to apply the Crandall-Rabinowitz theorem to the equation $G(\chi, \alpha) = 0$ at the bifurcation point $(0,0)$.

% \begin{theorem}\label{thm:Existence_bif_branches}
%     For each $n\in\N$, there exist $\delta>0$ and functions $\alpha : (-\delta, \delta) \to \R$, $z : (-\delta, \delta) \to X$ such that $\Gamma(s) = \Big(\bar{\Phi}_n + \alpha(s) + s \varphi_n + s z(s), \kappa_n + \alpha(s)\Big)$ is a branch of solutions to problem \eqref{equ_varphi_Stat} bifurcating from $(\bar{\Phi}_n, \kappa_n)$.
%     Furthermore, it holds $\alpha(0) = 0$, $z(0) = 0$ and $\varphi_n = \sqrt{2} \cos(2n\pi x)$ and
%     for $s\ne0$ small enough, these stationary solutions are nonconstant.
% \end{theorem}

\begin{theorem}
    \label{thm:BifurcationBranch}
    For each $n\in\N$ there exist $\delta >0$ and smooth functions
    \begin{AlignedEquation}
        \alpha_n(-\delta,\delta) \to \R, \quad z_n:(-\delta,\delta) \to X
    \end{AlignedEquation}
    with $\alpha_n(0) = z_n(0) = 0$, such that curve
    \begin{AlignedEquation}
        \Gamma_n(s) = \Big(\bar{\Phi}_n + \alpha_n(s) + s \varphi_n + s z_n(s), \kappa_n + \alpha_n(s)\Big)
    \end{AlignedEquation}
    forms a branch of solutions to problem \eqref{equ_varphi_Stat} bifurcating from $\big(\bar{\Phi}_n, \kappa_n\big)$. 
    For all $s\in(-\delta,\delta) \setminus \lbrace0\rbrace$, the solutions on $\Gamma_n(s)$ are nonconstant.
 \end{theorem}

\begin{proof}
    Let us verify the assumptions of \cite[Theorem I.5.1]{Kielhoefer}.
    % Theorem \ref{thm:BifSimplEigv}.
    The trivial branch condition holds as $G(0,\alpha) = 0$ for all $\alpha \in \mathbb{R}$.
    
    Next, we prove that $G_\chi(0,0)$ is a Fredholm operator of index zero with one-dimensional kernel.
    The linearization $G_\chi(0,0): X \to Y$ is given by
    \begin{AlignedEquation}
        \label{eq:GFunFirstDer}
        G_\chi(0,0)(h) = D h'' + (\kappa_n - 1) h - \kappa_n \int_0^1 h(y) \, dy.
    \end{AlignedEquation}
    The kernel is one-dimensional as $\ker G_\chi(0,0) = \ker\mathcal{L}(\kappa_n) = \text{span}\{\varphi_n\}$. To establish that $\text{codim ran } G_\chi(0,0) = 1$, we prove that $\text{ran } G_\chi(0,0) = \text{span}\{\varphi_n\}^\perp$.

    First, we show $\text{ran } G_\chi(0,0) \subseteq \text{span}\{\varphi_n\}^\perp$. For any $h \in X$, multiplying \eqref{eq:GFunFirstDer} by $\varphi_n$ and integrating:
    \begin{AlignedEquation}
        \int_0^1 G_\chi(0,0)(h) \cdot \varphi_n \, dy = \int_0^1 D h \varphi_n'' \, dy + (\kappa_n - 1) \int_0^1 h \varphi_n \, dy = 0,
    \end{AlignedEquation}
    where we used integration by parts, $\varphi_n'' = -\mu_n \varphi_n$, $\kappa_n - 1 = D\mu_n$, and $\int_0^1 \varphi_n \, dy = 0$.

    For the reverse inclusion, let $g \in Y$ with $\langle g, \varphi_n \rangle = 0$. We seek $h \in X$ such that $G_\chi(0,0)(h) = g$. Expanding in the eigenbasis of $\mathcal{L}(\kappa_n)$: $g = \sum_{j=0}^{\infty} g_j \varphi_j$ with $g_n = 0$, we obtain
    \begin{AlignedEquation}
        h = -g_0 + \sum_{j \neq n} \frac{g_j}{\kappa_n - 1 - D\mu_j} \varphi_j.
    \end{AlignedEquation}
    This is well-defined since $\kappa_n - 1 - D\mu_j = D(\mu_n - \mu_j) \neq 0$ for $j \neq n$.
    Furthermore, we only use the terms of the eigenbasis which satisfy the symmetry condition $\varphi(x) = \varphi(1-x)$ as $g\in Y$ and thus by regularization of $(\partial_{xx} -a)^{-1}$ we have $h\in X$.     Therefore, $\text{ran } G_\chi(0,0) = \text{span}\{\varphi_n\}^\perp$ has codimension one. Since the kernel is one-dimensional and the range has codimension one, $G_\chi(0,0)$ is a Fredholm operator of index zero.

    For the transversality condition of the Crandall-Rabinowitz theorem, observe that 
    \begin{AlignedEquation}
    G_{\chi \alpha}(0,0)(\varphi_n) = \varphi_n - \int_0^1 \varphi_n \, dy = \varphi_n \notin \operatorname{ran} G_\chi(0,0),    
    \end{AlignedEquation}
    since $\int_0^1 \varphi_n \, dy = 0$ for $n \geq 1$.
    
    Having verified all conditions (trivial branch, Fredholm property with index zero, one-dimensional kernel, and transversality condition), we obtain the existence of a branch $\big(\chi_n(s), \alpha_n(s)\big)$ of solutions satisfying $G(\chi_n(s), \alpha_n(s)) = 0$ for all $s \in (-\delta, \delta)$ with $\big(\chi_n(0), \alpha_n(0)\big) = (0,0)$. More precisely, there exists an open neighborhood $B_1 \subset B$ of $(0,0)$ and smooth functions $\alpha_n : (-\delta, \delta) \to \mathbb{R}$, $z_n : (-\delta, \delta) \to X$ with $\alpha_n(0) = z_n(0) = 0$ such that 
    \begin{AlignedEquation}
        B_1 \cap G^{-1}(0) = \{(s \varphi_n + s z_n(s), \alpha_n(s)) \mid -\delta < s < \delta\} \cup \{(0, s) \mid -\varepsilon < s < \varepsilon\}.
    \end{AlignedEquation}
    Consequently, $\Big(\bar\Phi_n + \alpha_n(s) + s \varphi_n + s z_n(s), \kappa_n + \alpha_n(s)\Big)$ constitutes a branch of solutions to problem \eqref{equ_varphi_Stat}.     Since $\varphi_n$ is nonconstant and $z_n(s) \to 0$ as $s \to 0$, these solutions are nonconstant for all $s \neq 0$ with sufficiently small $|s|$.
\end{proof}

\subsection{Direction of branches (sub- and supercritical pitchfork)}

To determine the direction of bifurcation, we analyze the behavior of the bifurcation parameter $\alpha_n(s)$ near $s = 0$.
Define the function
\begin{align*}
    \mathcal{G}(s) := G(s\varphi_n + sz_n(s), \alpha_n(s)) \equiv 0
\end{align*}
and differentiate it with respect to $s$ to extract the behavior of $\alpha_n(s)$ near the bifurcation point.

\begin{lemma}
    The branch of solutions $\Gamma_n(s) = \Big(\bar{\Phi}_n + \alpha_n(s) + s \varphi_n + s z_n(s), \kappa_n + \alpha(s)\Big)$ constructed in Theorem \ref{thm:BifurcationBranch} has the asymptotic expansion

\begin{AlignedEquation}
    \Gamma_n(s) \approx \Big(\kappa_n + \varphi_n s + \big(\alpha_n''(0) + z_n'(0)\big) s^2, \ \kappa_n + \alpha_n''(0) s^2\Big),
\end{AlignedEquation}
with
\begin{AlignedEquation} 
    z_n'(0) = \frac{\kappa_n}{24 \pi^2 n^2 D} \cos(4\pi n x)\quad \text{and} \quad \alpha_n''(0) = \frac{1}{4} - \frac{1}{16 D n^2 \pi^2} + 2 D n^2 \pi^2.
\end{AlignedEquation}
\end{lemma}

\begin{proof}
    We differentiate $\mathcal{G}(s)$ with respect to $s$ up to third order. The derivatives of $G$ with respect to $\chi$ and $\alpha$ are given in Appendix \ref{sec:deriv_G}.

    \textbf{First derivative:}
    Differentiating $\mathcal{G}(s)$ with respect to $s$, we obtain
    \begin{AlignedEquation}
        G_\chi(s\varphi_n + sz_n(s), \alpha_n(s))(\varphi_n + z_n(s) + sz_n'(s)) + G_\alpha(s\varphi_n + sz_n(s), \alpha_n(s))\alpha_n'(s) = 0.
    \end{AlignedEquation}
    Evaluating at $s = 0$ and using $z_n(0) = 0$, $\alpha_n(0) = 0$, and $G_\alpha(0,0) = 0$, we obtain
    \begin{AlignedEquation}
        G_\chi(0,0)(\varphi_n) = 0,
    \end{AlignedEquation}
    which holds since $\varphi_n \in \ker G_\chi(0,0)$. This does not determine $\alpha_n'(0)$, so we proceed to the second derivative.
    
    \textbf{Second derivative:}
    Differentiating twice and evaluating at $s = 0$, using $z_n(0) = 0$, $\alpha_n(0) = 0$, $\int \varphi_n \, dy = 0$, $G_\alpha(0,0) = 0$, and $G_{\alpha\alpha}(\chi, \alpha) \equiv 0$, we obtain
    \begin{AlignedEquation}
        0 &= G_{\chi\chi}(0,0)(\varphi_n, \varphi_n) + 2G_{\chi\alpha}(0,0)(\varphi_n)\alpha_n'(0) + 2G_\chi(0,0)(z_n'(0)) \\
        &= 2\big(Dz_n'(0)_{xx} + (\kappa_n - 1)z_n'(0)\big) - 2\kappa_n \int z_n'(0) \, dy + \kappa_n\left(\varphi_n^2 - \int \varphi_n^2 \, dy\right) + 2\varphi_n\alpha_n'(0).
    \end{AlignedEquation}
    Since $z_n'(0)$ is periodic, we obtain $\int z_n'(0) \, dy = 0$ and hence, the corresponding equation can be rewritten as
    \begin{AlignedEquation}
        \label{eq:GSecDerCon1}
        2(D z_n'(0)_{xx} + (\kappa_n -1) z_n'(0)) = -\kappa_n \left(\varphi_n^2 - \int \varphi_n^2 d y\right) - 2 \varphi_n \alpha_n'(0).
    \end{AlignedEquation}
    Multiplying equation \eqref{eq:GSecDerCon1} by $\varphi_n$, integrating over the domain and using the identities $\varphi_n'' = -\mu_n \varphi_n$ and $\kappa_n - 1 = D\mu_n$ together with $\int \varphi_n^3 = \int\varphi_n = 0$ we conclude
    % The solvability condition $F\perp \ker L ^* = \ker L = \mathrm{span}\{\varphi_n\}$ yields
    \begin{AlignedEquation}
        0  = - 2 \alpha_n'(0) \int \varphi_n^2 dy
    \end{AlignedEquation}
    which implies $\alpha_n'(0) = 0$. To find an explicit expression for $z_n'(0)$, we solve
    \begin{AlignedEquation}
        D z_n'(0)_{xx} + (\kappa_n - 1) z_n'(0) &= - \frac{\kappa_n}{2} \left(\varphi_n^2 - \int \varphi_n^2 dy\right) = -  \frac{\kappa_n}{2} \cos(4n\pi x),
    \end{AlignedEquation}
    where we used $\varphi_n(x)^2 = 2\cos^2(2\pi nx) = 1 + \cos(4\pi nx)$ and $\int \varphi_n^2 \, dy = 1$.
    The solution has the explicit form
    \begin{AlignedEquation}
        z_n'(0) = A \cos(4n\pi x), \quad \text{with } - D A (4n\pi)^2 + (\kappa_n -1) A = -  \frac{\kappa_n}{2}.
    \end{AlignedEquation}
    Solving for $A$ and using $\kappa_n-1 = D\mu_n = D 4 \pi^2n^2$ gives
    $A = {\kappa_n}({24 n^2 \pi^2 D})^{-1}$.
    
    \textbf{Third derivative:}
    Differentiating $\mathcal G(s)$ three times and evaluating at $s=0$, noting that $\int z_n''(0) dy = 0$, in analogy with $\int z_n'(0) dy = 0$ and using the properties derived from the second derivative, we obtain
    \begin{AlignedEquation}
    \label{eq:G3Der}
        0 &= 3 (D z_n''(0)_{xx} + (\kappa_n - 1) z_n''(0)) + \kappa_n \left(\varphi_n^3 - 3\varphi_n \int \varphi_n^2 dy\right) \\
        &+ 6 \kappa_n \left(\varphi_n z_n'(0) - \int \varphi_n z_n'(0) d y\right)  + 2 \alpha_n''(0) \varphi_n.
    \end{AlignedEquation}
    %  we arrive at the following solvability condition, which by orthogonality requires that $F\perp \varphi_n$, where
    % \begin{AlignedEquation}
    %     \label{eq:G3Der}
    %     F := \kappa_n \left(\varphi_n^3 - 3\varphi_n \int \varphi_n^2 dy\right) + 6 \kappa_n \left(\varphi_n z'(0) - \int \varphi_n z'(0) d y\right) + 2 \alpha''(0) \varphi_n.
    % \end{AlignedEquation}
    Analogously, we multiply equation \eqref{eq:G3Der} by $\varphi_n$, integrate, and use indentities 
    \begin{AlignedEquation}
        \int_0^1 \varphi_n^2 dy = 1\quad\int_0^1 \varphi_n^4 dy = 3/2\quad \int_0^1 \varphi_n^2 z_n'(0) dy = 1/2 A 
    \end{AlignedEquation}
    which by direct calculation yields
    \begin{AlignedEquation}
        0 &=  3/2 \kappa_n - 3\kappa_n + 3 \kappa_n A + 2 \alpha_n''(0). 
    \end{AlignedEquation}
    Substituting $A = {\kappa_n}({24 \pi^2 n^2 D})^{-1}$ and $\kappa_n = 1 + D4\pi^2n^2$, we obtain
    \begin{AlignedEquation}
        \label{eq:FormulaAlphaPP}
        \alpha_n''(0) = \frac{1}{4} - \frac{1}{16 D n^2 \pi^2} + 2 D n^2 \pi^2.
    \end{AlignedEquation}
\end{proof}

Having derived an explicit formula for $\alpha_n''(0)$, we can now classify the type of bifurcation. We emphasize that this classification applies only to the local behavior near the bifurcation point and is determined solely by the eigenvalue that crosses zero. Other eigenvalues in the spectrum, although not involved in the bifurcation mechanism, do not affect this local characterization. Importantly, the bifurcation type by itself does not determine the overall linear stability of the bifurcating branch; a full stability assessment requires examining the entire spectrum of the linearized operator along the bifurcating solutions.
% \begin{theorem}
%     Fix $n\in\N$.
%     If $\kappa_n\ne1.5$, then $(\bar\Phi_n,\kappa_n)$ exhibits a pitchfork bifurcation.
%     Using the notation of Theorem \ref{thm:Existence_bif_branches}, the sign of $\alpha''(0)$ determines the type of the pitchfork bifurcation:
%     \begin{itemize}
%         \item Subcritical if $\alpha''(0) < 0$, which occurs when
%         \begin{align*}
%             0 < D < \frac{1}{8 n^2 \pi^2}, \quad \text{equivalently } \kappa_n < 1.5,
%         \end{align*}
%         \item Supercritical if $\alpha''(0) > 0$, i.e., $D> \frac{1}{8 n^2 \pi^2}$ or equivalently $\kappa_n > 1.5$.
%     \end{itemize}
% \end{theorem}
\begin{theorem}
    \label{thm:BifurcationType}
    Fix $n \in \mathbb{N}$. If $\kappa_n \neq 1.5$, then $(\bar\Phi_n, \kappa_n)$ undergoes a pitchfork bifurcation. The type of bifurcation is determined by the sign of $\alpha_n''(0)$:
    \begin{itemize}
        \item subcritical if $\alpha_n''(0) < 0$, i.e., $\kappa_n < 1.5$ $($equivalently, $0 < D < (8 n^2 \pi^2)^{-1})$,
        \item supercritical if $\alpha_n''(0) > 0$, i.e., $\kappa_n > 1.5$ $($equivalently, $D > (8 n^2 \pi^2)^{-1})$.
    \end{itemize}
\end{theorem}

\begin{proof}
    Observe that by formula \eqref{eq:FormulaAlphaPP} there holds $\alpha_n''(0) = 0$ if and only if $D = \frac{1}{8 n^2 \pi^2}$, which corresponds to $\kappa_n = 1.5$. Thus, whenever $\alpha_n'(0) = 0$ and $\alpha_n''(0) \neq 0$, a pitchfork bifurcation occurs at $\bar\Phi_n$ for $\kappa = \kappa_n$.

    To determine the type of bifurcation, we examine the behavior of the trivial branch. Since the critical eigenvalue responsible for bifurcation is $\kappa - \kappa_n$, the trivial branch loses stability as $\kappa$ crosses $\kappa_n$. By the principle of exchange of stability, the bifurcation is subcritical if the nontrivial branch enters the region where the trivial solution is stable, i.e., if $\kappa_n + \alpha_n(s) < \kappa_n$ for small $s > 0$. Conversely, it is supercritical if $\kappa_n + \alpha_n(s) > \kappa_n$. Using the expansion
    $\alpha_n(s) = \tfrac{1}{2} \alpha_n''(0) s^2 + \mathcal  O(s^3)$,
    it follows that $\alpha_n''(0) < 0$ implies a subcritical bifurcation, while $\alpha_n''(0) > 0$ implies a supercritical one. The equivalence between the sign of $\alpha_n''(0)$ and the conditions on $D$ and $\kappa_n$ follows from direct computation.
\end{proof}
\noindent To conclude about the stability of the branch, we consider the bifurcation point $\kappa_1$, namely a point at which the trivial branch changes stability.

% \begin{theorem}
%     Fix $(\bar\Phi_n,\kappa_n)$ for some $n \in \mathbb{N}$, and let $\Gamma_n(s)$ denote the bifurcating branch. Then, for $s$ sufficiently small, the following holds:
%     \begin{itemize}
%         \item For $n \geqslant 2$, the branch $\Gamma_n(s)$ is unstable;
%         \item For $n = 1$ and $\kappa_1 < 1.5$, the branch $\Gamma_1(s)$ is unstable;
%         \item For $n = 1$ and $\kappa_1 > 1.5$, the branch $\Gamma_1(s)$ is stable.
%     \end{itemize}
% \end{theorem}

\begin{theorem}\label{thm:stability_bifurcating_branches}
    Fix $(\bar{\Phi}_n, \kappa_n)$ for some $n \in \mathbb{N}$, and let $\Gamma_n(s) = (\Phi_n(s), \kappa_n(s))$ denote the bifurcating branch. Then, for sufficiently small $|s|$, the following stability properties hold:
    \begin{itemize}
        \item For $n \geq 2$, the branch $\Gamma_n(s)$ consists of linearly unstable solutions;
        \item For $n = 1$ and $\kappa_1 < 1.5$, the branch $\Gamma_1(s)$ consists of linearly unstable solutions;
        \item For $n = 1$ and $\kappa_1 > 1.5$, the branch $\Gamma_1(s)$ consists of linearly stable solutions.
    \end{itemize}
\end{theorem}

\begin{proof}
    At the bifurcation point $(\bar{\Phi}_n, \kappa_n)$ with $n \geq 2$, Proposition~\ref{thm:Stab_constant_steady_state} shows that the linearized operator $\mathcal{L}(\kappa_n)$ possesses exactly $n-1$ positive eigenvalues in addition to the simple zero eigenvalue. By the continuation principle, these positive eigenvalues persist along the bifurcating branch $\Gamma_n(s)$ for sufficiently small $|s|$. Consequently, regardless of the sign of the critical eigenvalue along the branch, the solutions $\Phi_n(s)$ remain linearly unstable in a neighborhood of the bifurcation point.

     For the primary bifurcation at $(\bar{\Phi}_1, \kappa_1)$, the linearized operator $\mathcal{L}(\kappa_1)$ has a simple zero eigenvalue and all other eigenvalues are negative. The stability of the bifurcating branch $\Gamma_1(s)$ is therefore determined exclusively by the behavior of the critical eigenvalue along the nontrivial branch. The assertion follows from an analogous argument to that employed in Theorem~\ref{thm:BifurcationType}.
\end{proof}

\subsection{Global continuation and turning points}

Fix a bifurcation point $(\bar{\Phi}_n, \kappa_n)$ for some $n \in \mathbb{N}$. We now establish that the local branch $\Gamma_n(s)$ obtained in Theorem~\ref{thm:BifurcationBranch} extends globally.

To apply the global bifurcation theory of Rabinowitz, we reformulate the steady-state equation \eqref{equ_varphi_Stat} as a fixed-point problem. Specifically, \eqref{equ_varphi_Stat} is equivalent to
\begin{equation}
    0 = \varphi + (D\partial_{xx} - I)^{-1}\left(\kappa\frac{e^{\varphi}}{\int_{\mathbb{T}} e^{\varphi}\,dy}\right) =: \varphi + \kappa f(\varphi) =: H(\varphi, \kappa).
\end{equation}

Let $\mathcal{S}$ denote the closure of nontrivial solutions to $H(\varphi, \kappa) = 0$ in $W^{1,2}(\mathbb{T}) \times \mathbb{R}$. Note that $\Gamma_n(s) \subset \mathcal{S}$, and by elliptic regularity, $\mathcal{S} \subset W^{2,2}(\mathbb{T}) \times \mathbb{R}$.

\begin{theorem}[Global bifurcation]\label{thm:global_branches}
    Let $\mathcal{C}$ be the maximal connected component in $\mathcal{S}$ containing $(\bar{\Phi}_n, \kappa_n)$. Then $\mathcal{C}$ satisfies one of the following alternatives:
    \begin{enumerate}
        \item[\rm(i)] $\mathcal{C}$ is unbounded in $W^{1,2}(\mathbb{T}) \times \mathbb{R}$;
        \item[\rm(ii)] $\mathcal{C}$ contains a point $(\bar{\Phi}(\kappa), \kappa)$ on the trivial branch for some $\kappa \neq \kappa_n$.
    \end{enumerate}
\end{theorem}

\begin{proof}
    We verify the assumptions of the Global Crandall-Rabinowitz Theorem \cite[special case of Theorem II.3.3]{Kielhoefer}.
    % (Theorem~\ref{thm:Global_Crandall_Rabinowitz})
    Since $W^{1,2}(\mathbb{T}) \subseteq L^{\infty}(\mathbb{T})$, we have $e^{\varphi} \in L^{\infty}(\mathbb{T})$ for all $\varphi \in W^{1,2}(\mathbb{T})$. By elliptic regularity and the compact embedding $W^{2,2}(\mathbb{T}) \subseteq W^{1,2}(\mathbb{T})$, the operator
    \begin{AlignedEquation}
        (D\partial_{xx} - I)^{-1}: L^2(\mathbb{T}) \to W^{2,2}(\mathbb{T}) \subseteq W^{1,2}(\mathbb{T})
    \end{AlignedEquation}
    is compact. Therefore, $f: W^{1,2}(\mathbb{T}) \to W^{1,2}(\mathbb{T})$ is well-defined and completely continuous.
    
    To establish continuity of the linearization, we compute $D_{\varphi}H(\bar{\Phi}_n, \kappa) = I + \kappa D_{\varphi}f(\bar{\Phi}_n)$. At the constant state $\bar{\Phi}_n \equiv \kappa_n$, the linearization yields
    \begin{AlignedEquation}
        D_{\varphi}f(\bar{\Phi}_n)\psi = (D\partial_{xx} - I)^{-1}\left(\psi - \int_{\mathbb{T}} \psi\,dy\right).
    \end{AlignedEquation}
    Standard elliptic estimates and Sobolev embedding give
    \begin{AlignedEquation}
        |D_{\varphi}f(\bar{\Phi}_n)\psi\|_{W^{1,2}} \leq C\left\|\psi - \int_{\mathbb{T}} \psi\,dy\right\|_{L^2} \leq C\|\psi\|_{W^{1,2}},
    \end{AlignedEquation}
    which proves $D_{\varphi}H(\bar{\Phi}_n, \cdot) \in C(\mathbb{R}, \mathcal{L}(W^{1,2}(\mathbb{T}), W^{1,2}(\mathbb{T})))$.
    
    Finally, the analysis in Theorem~\ref{thm:BifurcationBranch} shows that $\ker(D_{\varphi}H(\bar{\Phi}_n, \kappa_n)) = \operatorname{span}\{\varphi_n\}$ with $\varphi_n = \sqrt{2}\cos(2n\pi x)$, and $\varphi_n \notin \operatorname{ran}(D_{\varphi}H(\bar{\Phi}_n, \kappa_n))$. Indeed,
    \begin{AlignedEquation}
        G_\chi(0,0)\varphi = 0 \ &\Longleftrightarrow \ D_\varphi H(\bar\Phi_n, \kappa_n) \varphi = 0, \\
        \varphi \in \operatorname{ran}(G_\chi(0,0)) \ &\Longleftrightarrow \ (D\Delta - I)^{-1}\varphi \in \operatorname{ran}(D_\varphi H(\bar\Phi_n, \kappa_n)),
    \end{AlignedEquation}
    and thus zero is a simple eigenvalue of $D_{\varphi}H(\bar{\Phi}_n, \kappa_n)$.
    
    Having verified all assumptions, the conclusion follows. 
    % from Theorem~\ref{thm:Global_Crandall_Rabinowitz}.
\end{proof}

We now analyze the global behavior of the bifurcating branch in the subcritical regime. 

\begin{theorem}[Existence of turning point]\label{thm:fold_bifurcation}
    Suppose $\kappa_1 = 1 + D\mu_1 < 1.5$, corresponding to a subcritical pitchfork bifurcation at $(\bar{\Phi}_1, \kappa_1)$. Then there exists $\hat{s} \in \mathbb{R}$ such that the branch $\Gamma_1(s)$ exhibits a turning point at $(\Phi(\hat{s}), \kappa_f)$, where $\kappa_f = \kappa_1(\hat{s}) \in (0, \kappa_1)$.
\end{theorem}

\begin{figure}[t]
    \centering
\includegraphics[width=0.4\linewidth]{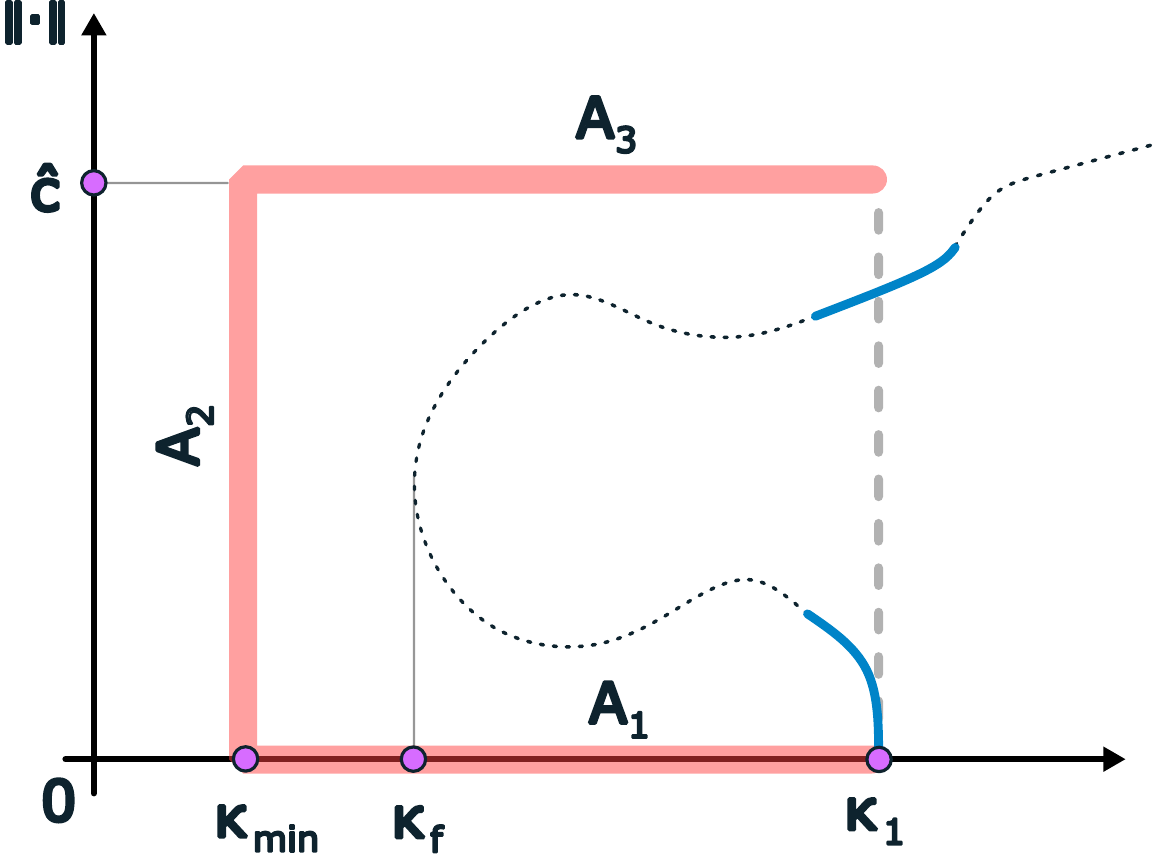}
    \caption{}
    \label{fig:branch_leaves_to_right_side}
    \end{figure}

\begin{proof}
We analyze the possible behavior of $\Gamma_1(s)$. By Theorem~\ref{thm:global_branches}, the maximal connected component containing $\Gamma_1(s)$ either extends to infinity in $W^{1,2}(\mathbb{T}) \times \mathbb{R}$ or returns to the trivial branch at some $\kappa \neq \kappa_1$. Multiplying equation~\eqref{equ_varphi_Stat} by $U$ and integrating yields
\begin{equation}
    D\int_{\mathbb{T}} |U_x|^2\,dx + \int_{\mathbb{T}} U^2\,dx = \kappa\frac{\int_{\mathbb{T}} e^U U\,dx}{\int_{\mathbb{T}} e^U\,\dy} \leq \kappa\|U\|_{L^{\infty}(\mathbb{T})}.
\end{equation}
Therefore, by Sobolev embedding 
$\|U\|_{W^{1,2}}^2 \leq C\kappa\|U\|_{W^{1,2}}$, and hence all stationary solutions are uniformly bounded in $W^{1,2}$ for $\kappa \in[0, K]$. In particular, the branch $\Gamma_1(s)$ can not cross the horizontal line $A_3$ in Figure~\ref{fig:branch_leaves_to_right_side}.

Since $\kappa_1$ is the smallest bifurcation point, the branch cannot intersect the trivial branch for $\kappa \in (0, \kappa_1)$; this corresponds to the horizontal line $A_1$ in Figure~\ref{fig:branch_leaves_to_right_side}. Furthermore, Theorem~\ref{thm:global_branches} guarantees the uniqueness of the constant solution for $\kappa < \kappa_{\min}$, which forms the boundary $A_2$.

Since the branch emerges from $(\bar{\Phi}_1, \kappa_1)$ with $\kappa_1'(0) = 0$ and $\kappa_1''(0) < 0$ (subcritical bifurcation), and remains confined within the bounded rectangle enclosed by $A_1$, $A_2$, and $A_3$, Theorem~\ref{thm:global_branches} implies that the branch must exit this region. The only possible exit is through right boundary (gray dashed line).

By continuity there exists $\hat{s} \in \mathbb{R}$ which is a turning point of the branch $\Gamma_1(s)$.
\end{proof}

\begin{remark}
We cannot analytically prove how the branch in $[\kappa_{\min}, \kappa_1]$ behaves in detail.
Numerical simulations suggest that we always have the simplest case of one fold bifurcation at $\kappa_f$.
At the fold bifurcation $\kappa_f$ the unstable eigenvalue from the subcritical pitchfork bifurcation could get stable or one of the stable eigenvalues could loose stability.
In the first case, the branch would be stable after the fold bifurcation and we would have the existence of stable nontrivial steady states for some $\kappa \in(\kappa_f,\kappa_0)$ and, thus, bistability as the constant steady state is also stable.
This would not be necessarily true in the second case.
However, in numerical analysis, we always observe a single fold bifurcation $\kappa_f$ in $(0, \kappa_1)$ at which the branch $\Gamma(s)$ gets stable. Hence, we have bistability in the whole interval $(\kappa_f, \kappa_1)$.
\end{remark}

\bibliographystyle{abbrv}
\bibliography{sample}

\appendix

\newpage

\section{Model derivation}
\label{appsec:Deriv}

\begin{figure}[b]
	\centering
	\includegraphics[width=0.6\linewidth]{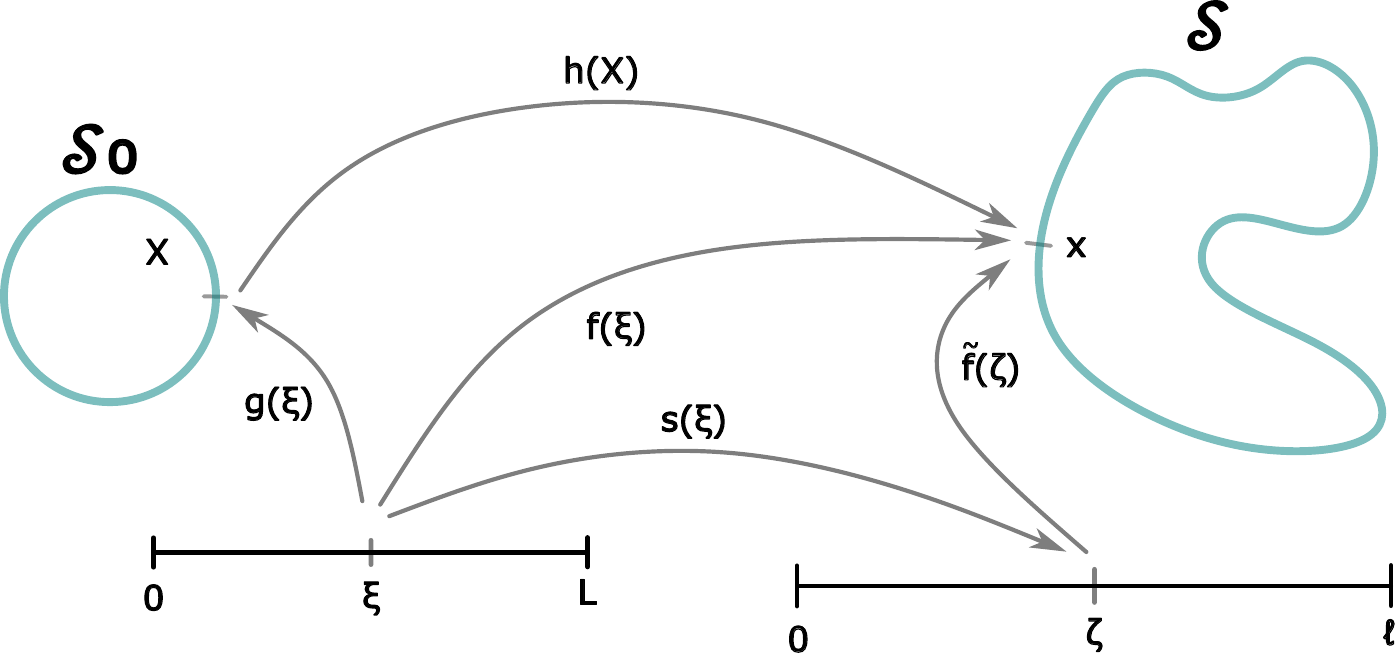}
	\caption{Schematic representation of the elastic deformation framework for the {\it Hydra} spheroid model. Left: Reference configuration $\mathcal{S}_0$ parametrized by $\mathbf{g}:[0,L] \to \mathcal{S}_0$.  Right: Deformed configuration $\mathcal{S}$ showing the composition of mappings $\mathbf{f} =\mathbf{h} \circ \mathbf{g}$ and the arc-length reparametrization $\mathbf{\tilde f}$ with $\mathbf{f} = \mathbf{\tilde f} \circ s$. }
	\label{fig:sketch2d_framework}
\end{figure}

\subsection{Hyper-elastic tissue model}
We demonstrate the derivation of the mechanochemical patterning model for \emph{Hydra} spheroids introduced in~\cite{Weevers2025} in two spatial dimensions. It retains the essential coupling between mechanics and chemistry while simplifying the geometric structure. In this setting, the \emph{Hydra} spheroid is represented as an elastic closed curve, providing a minimal framework to analyze the core mechanochemical feedback mechanisms, see Figure~\ref{fig:sketch2d_framework}.

The initial, stress-free reference configuration, denoted by $\mathcal S_0$, is assumed to be a circle of fixed circumference $L>0$, parametrized by
\begin{AlignedEquation}
	\label{eq:reference_curve}
	\mathbf{g}(\xi) = \frac{L}{2 \pi} 
	\begin{pmatrix}
		\cos\frac{2 \pi}{L} \xi\\
		\sin\frac{2\pi}{L}\xi
	\end{pmatrix}, \quad \xi \in [0,L].
\end{AlignedEquation}
This reference curve is subjected to a spatial deformation $\mathbf{h}:\mathcal S_0 \to \mathcal S$, resulting in a new closed curve $\mathcal S$ with circumference $\ell>0$. We prescribe a fixed enclosed area $A_0>\operatorname{Area}(\mathcal S_0)$ and restrict attention to deformations that satisfy $\operatorname{Area}(\mathcal S) = A_0$. The deformed curve is parametrized by the composition
\begin{AlignedEquation}
	\label{eq:deformed_curve}
	\mathbf{f}:[0, L]\to \mathcal S, \quad \mathbf{f} = \mathbf{h} \circ \mathbf{g}.
\end{AlignedEquation}

The deformation gradient in surface coordinates is given by the matrix
\begin{AlignedEquation}
	\label{eq:deformation_gradient}
	\mathbf{F} = \nabla_\mathbf{X} \mathbf{h}(\mathbf{X}) = \mathbf{f}'(\xi) \otimes \nabla \big(\mathbf{g}^{-1}(\mathbf{X})\big), \quad \mathbf{X} = \mathbf{g}(\xi) \; .
\end{AlignedEquation}
Taking a derivative of the identity $\mathbf{g}^{-1}(\mathbf{g}(\xi))=\xi$ and using $|\mathbf{g}'(\xi)| \equiv 1$ implies that  $\nabla(\mathbf{g}^{-1}(\mathbf{X})) = \mathbf{g}'(\xi)^T$, where $\mathbf{g}'(\xi)$ (the covariant basis vector on $\mathcal{S}_0$) is a column vector and the contravariant basis vector  $\mathbf{g}'(\xi)^T$ is a row vector.
We therefore can write 
\begin{AlignedEquation}
	\label{eq:deformation_gradient2}
	\mathbf{F} = \mathbf{f}'(\xi)  \otimes \mathbf{g}'(\xi)^T.
\end{AlignedEquation}
The \textit{Right Cauchy-Green} deformation tensor reads
\begin{AlignedEquation}
	\label{eq:cg_def_tensor}
	\mathbf{C} = \mathbf{F}^T \mathbf{F} =  \left(\mathbf{g}'(\xi) \otimes \mathbf{f}'(\xi)^T \right) \left( \mathbf{f}'(\xi)  \otimes \mathbf{g}'(\xi)^T\right)= 
	|\mathbf{f}'(\xi)|^2 \, \mathbf{g}'(\xi) \otimes \mathbf{g}'(\xi)^T\; ,
\end{AlignedEquation}
and the Green-Lagrangian strain is
\begin{AlignedEquation}
	\label{eq:StrainFormula}
	\boldsymbol{\varepsilon} = \frac12(\mathbf{C}-\mathbf{I}) = \frac12\left(|\mathbf{f}'(\xi)|^2 - 1 \right) \mathbf{g}'(\xi) \otimes  \mathbf{g}'(\xi)^T,
\end{AlignedEquation}
where the identity in the tangent space of $S_0$ is given by $\mathbf{I}=\mathbf{g}'(\xi) \otimes \mathbf{g}'(\xi)^T$. 

We model the  tissue as a hyperelastic material governed by the Saint Venant-Kirchhoff constitutive law with strain energy density
\begin{AlignedEquation}
	W(\boldsymbol{\varepsilon}) = \frac{\lambda}{2} {\rm tr}^2(\boldsymbol{\varepsilon}) + \frac{\mu}{2}{\rm tr}(\boldsymbol{\varepsilon}^2),
\end{AlignedEquation}
where the Lamé parameters satisfy $\lambda = 0$ and $\mu = E(u)$, with $u:=u(\xi)$ the morphogen concentration and $E(u)$ the morphogen-dependent elastic modulus encoding the mechanochemical coupling central to our model. The specific choice of parameters follows from the modelling considerations in \cite{Weevers2025}.
The total elastic energy of a deformation is given by
\begin{AlignedEquation}
	\mathcal E_{\text{el}} = \int_{S_0} W(\varepsilon) \, dS_0 = \frac{1}{8}\int_0^L{E(u)} \left(|\mathbf{f}'(\xi)|^2 - 1 \right)^2 |\mathbf{g}'(\xi)| d\xi,
\end{AlignedEquation}
where we used equation \eqref{eq:StrainFormula}.

Assuming quasi-static mechanics due to fast tissue relaxation, the elastic energy is minimized
under the area constraint
\begin{AlignedEquation}
	A_0=A(f) = \frac12 \int_S \mathbf{x} \cdot \mathbf{n} \, dS = \frac12 \int_0^L \mathbf{f}(\xi)  \cdot \mathbf{n}(\mathbf{f}(\xi)) \, |\mathbf{f}'| \ {d} \xi = -\frac12 \int_0^L \mathbf{f}(\xi)  \cdot  \mathbf{f}'(\xi)^\perp \ {d} \xi \; , 
\end{AlignedEquation}
where $\mathbf{n}(\mathbf{x})=-\mathbf{f}'(\mathbf{x})^\perp/|\mathbf{f}'(\mathbf{x})|$ is the outer unit normal vector on $\mathcal{S}$ at $\mathbf{x} \in \mathcal{S}$ and ${}^\perp$ represents the rotation of a 2-dimensional vector by $\pi/2$: $(a,b)^\perp=(-b,a)$.

To find minimizers of $\mathcal{E}_{\text{el}}$ we introduce Lagrange multipliers. The Gateaux derivative with respect to variations $\delta \mathbf{f}(\xi)$ with periodic boundary conditions yields
\begin{AlignedEquation}
	\label{eq:EnergyGateuaxDerivative}
	0&=D \mathcal E_{\text{el}}(\mathbf{f}) (\delta \mathbf{f}) + p D A(f)(\delta f)\\
	%&=\frac{1}{2}\int_0^L E(u) \left(\frac{|\mathbf{f}'(\xi)|^2}{|\mathbf{g}'(\xi)|^2} - 1 \right) \frac{\mathbf{f}'(\xi) \cdot \delta \mathbf{f}'(\xi)}{|\mathbf{g}'(\xi)|^2}   |\mathbf{g}'(\xi)| d\xi- p \frac12 \int_0^L \left(\delta \mathbf{f}(\xi)  \cdot \mathbf{f}'(\xi)^\perp + \mathbf{f}(\xi)  \cdot \delta \mathbf{f}'(\xi)^\perp \right) d\xi\\
	&=\frac{1}{2}\int_0^L E(u) \left(|\mathbf{f}'(\xi)|^2 - 1 \right) \mathbf{f}'(\xi) \cdot \delta \mathbf{f}'(\xi)   d\xi- p  \int_0^L \delta \mathbf{f}(\xi)  \cdot \mathbf{f}'(\xi)^\perp  \, d\xi \; ,
\end{AlignedEquation}
where $p \in \mathbb{R}$, the inner pressure, acts as a Lagrange multiplier for the area constraint. Integrating by parts, we obtain that equation \eqref{eq:EnergyGateuaxDerivative} is satisfied if and only if the minimizer $f$ satisfies
\begin{AlignedEquation}\label{eq:MinimizerE_el}
	   \left(\sigma(\xi) \frac{\mathbf{f}'(\xi)}{|\mathbf{f}'(\xi)|} \right)'  + p \,  \mathbf{f}'(\xi)^\perp =0 \; ,
\end{AlignedEquation}
where the stress is given by
\begin{AlignedEquation}\label{eq:stressS}
	\sigma(\xi)=E(u) \frac{1}{2} \left(|\mathbf{f}'(\xi)|^2 - 1 \right) |\mathbf{f}'(\xi)| .
\end{AlignedEquation}	
Evaluating the derivative we find that 
\begin{AlignedEquation}\label{eq:MinimizerE_el2}
	\sigma'(\xi) \frac{\mathbf{f}'(\xi)}{|\mathbf{f}'(\xi)|}  +\sigma(\xi) \left( \frac{\mathbf{f}'(\xi)}{|\mathbf{f}'(\xi)|} \right)' + p \,  \mathbf{f}'(\xi)^\perp =0 \; .
\end{AlignedEquation}

Taking the scalar product of this with $\mathbf{f}'(\xi)$ and using that $\left( \frac{\mathbf{f}'(\xi)}{|\mathbf{f}'(\xi)|} \right)' \cdot \mathbf{f}'(\xi)=0$ as well as the impenetrability assumption $|\mathbf{f}'(\xi)|>0$ we obtain that $\sigma' \equiv 0$, i.e. the stress is constant. Taking the scalar product of \eqref{eq:MinimizerE_el2} with $\mathbf{f}'(\xi)^\perp$ we obtain that $\sigma \left( \frac{\mathbf{f}'(\xi)}{|\mathbf{f}'(\xi)|} \right)' \cdot \mathbf{f}'(\xi)^\perp+ p \, |\mathbf{f}'(\xi)|^2 =0$, i.e. the curvature $\kappa=\frac{\mathbf{f}''(\xi)}{|\mathbf{f}'(\xi)|^3} \cdot \mathbf{f}'(\xi)^\perp=-p/\sigma$ (Laplace's Law) is constant along $\mathcal{S}$.
Hence, the minimizer $\mathcal{S}$ must be circular with circumference $\ell > L$.

To separate geometric shape from tangential stretching, we introduce an arc-length parametrization:
\begin{AlignedEquation}
	\label{eq:arc_length_param}
	\mathbf{f}(\xi) = \mathbf{\tilde f}(s(\xi)),
\end{AlignedEquation}
where $\mathbf{\tilde f}:[0,\ell]\to\mathcal S$ satisfies $|\mathbf{\tilde f}'|=1$, and $s:[0,L]\to[0,\ell]$ with $s'>0$ encodes the tangential deformation. Since $s'(\xi)=|f'(\xi)|$ represents the local stretch ratio (ratio of deformed to reference arc length $|g'({\xi})| \equiv 1$), the quantity
\begin{equation}
	\label{eq:1D_strain}
	\tilde{\varepsilon} = \frac{1}{2}((s')^2 - 1)
\end{equation}
is the finite Green-Lagrangian strain in one dimension. 

Conservation of total length imposes
\begin{AlignedEquation}
	\label{eq:length_conservation}
	\int_0^L (s'(\xi)-1)\, d\xi = \ell - L.
\end{AlignedEquation}
The morphogen concentration $u(t,\xi)$ evolves according to
\begin{AlignedEquation}
	\label{eq:morphogen}
	\partial_t u = \kappa \tilde\varepsilon - \alpha u + D \partial_{\xi\xi} u,
\end{AlignedEquation}
with production proportional to local strain, linear degradation rate $\alpha>0$, and diffusion coefficient $D>0$. Taking together, the full mechanochemical system reads
\begin{equation}
	\label{eq:full_system}
	\left\{
	\begin{aligned}
		&\begin{aligned}
			\partial_t u &= \kappa \tilde\varepsilon - \alpha u + D \partial_{\xi\xi} u,\\
			0 &= \partial_\xi \big(E(u) (s'^2-1) s'\big),
		\end{aligned} && \xi \in [0,L],\ t\ge0,\\
		& \ell - L = \int_0^L (s'-1)\, d\xi, && t\ge0,
	\end{aligned}    
	\right.
\end{equation}
with periodic boundary conditions and appropriate initial data. The 1D finite strain is defined by \eqref{eq:1D_strain}, completing the mechanochemical coupling.

\subsection{Inifinitesimal strain}
\label{subsec:Infinit}
To facilitate analytical treatment of the mechanochemical model, we invoke the infinitesimal strain approximation. This is justified biologically, as \emph{Hydra} cells experience only small-amplitude distortion during mechanochemical signalling, so that the local deformation gradients remain close to the identity \cite{Weevers2025}. Under the assumption that 
\(
|s' - 1| \ll 1,
\) 
the finite strain 
\(
\tilde{\varepsilon} = \tfrac12(s'^2 - 1)
\) 
can be replaced by its linearized counterpart
\begin{AlignedEquation}
	\hat\varepsilon(t,\xi) := s'(t,\xi) - 1 \approx \tilde\varepsilon(t,\xi),
\end{AlignedEquation}
reducing the nonlinear system to a tractable linear framework while retaining the essential mechanochemical coupling.

Following the minimization procedure, the mechanical equilibrium condition reduces under the infinitesimal strain approximation to
\begin{AlignedEquation}
	0 = \frac{1}{2} \partial_\xi \Big(E(u) \, (s'^2 - 1)s'\Big)
	= \frac{1}{2} \partial_\xi \Big(E(u) \, \big(2 (s'-1) + (s'-1)^2\big) s'\Big)
	\approx \partial_\xi \big(E(u) \, \hat\varepsilon\big),
\end{AlignedEquation}
where $\hat\varepsilon := s'-1$ is the linearized strain.  
Consequently, the full mechanochemical system in the infinitesimal strain approximation, with $\varepsilon := \hat\varepsilon(\xi,t)$, reads
\begin{equation} \label{app:sys_eps}
	\left\{
	\begin{aligned}
		&\begin{aligned}        
			\partial_t u &= \kappa \, \varepsilon - \alpha u + D \, \partial_{\xi\xi} u,\\
			0 &= \partial_\xi \big(E(u) \, \varepsilon\big),
		\end{aligned} 
		&& \xi \in [0,L], \quad t \ge 0,\\
		& l - L = \int_0^L \varepsilon \, d\xi, && t \ge 0.
	\end{aligned}
	\right.    
\end{equation}

\section{Well-posedness of the model}
We provide here a proof of existence and uniqueness of a global-in-time solution to problem \eqref{equ_varphi}. A similar approach can be found in \cite{miyasita2007dynamical}; we include the details here for the sake of completeness.

\begin{proof}[Proof of proposition \ref{prop:ExistenceGlobal}]
    We begin by establishing the existence of a local-in-time solution. Let $u_0\in W^{1,2}(\T)$ and consider the mild formulation of problem \eqref{equ_varphi}:
    \begin{AlignedEquation}
        u(t,\cdot) = e^{D\Delta t} u_0 + \int_0^t e^{D\Delta(t-s)}\left( \kappa\frac{e^{u(s,\cdot )}}{\intT e^{u(s,y)}\dy} - u(s,\cdot) \right) ds \; .
    \end{AlignedEquation}
    Here, $e^{D\Delta t}$ is a semigroup generated on $L^2(\mathbb{T})$ by the second derivative $D\Delta$ with domain $D(D\Delta) = W^{2,2}(\mathbb{T})$. Let $X = {C\big([0,\tau];W^{1,2}(\mathbb{T})\big)}$ be the Banach space with the norm $\|v\|_X = \sup\limits_{0\leqslant t\leqslant \tau} \|v(\cdot, t)\|_{W^{1,2}}$ where $\tau$ is a positive constant to be determined further below. We define the map $F:X\to X$ by
    \begin{AlignedEquation}
        F(v) := e^{D\Delta t} u_0 + \int_0^t e^{D\Delta(t-s)}\left( \kappa\frac{e^{v(s, \cdot)}}{\intT e^{v(s,y)}dy} - v(s, \cdot) \right) ds.
    \end{AlignedEquation}
    We further introduce the ball 
    \begin{AlignedEquation}
        V = \left\lbrace v \in X : \|v\|_X\leqslant 3\delta\right\rbrace \quad \text{where} \quad \delta = \|u_0\|_{W^{1,2}}.  
    \end{AlignedEquation} 
    Notice that by Sobolev embedding, any $v \in V$ satisfies $\|v(t,\cdot)\|_{L^\infty} \leqslant C \delta$ for a constant $C>0$. 
    It is well known that the semigroup $e^{D\Delta t}$ satisfies  the estimates
    \begin{AlignedEquation}
        \|e^{D\Delta t} v\|_{L^2} \leqslant C \|v\|_{L^2} \quad \text{and} \quad \|\nabla e^{D\Delta t} v\|_{L^2} \leqslant C t^{-\frac{1}{2}} \|v\|_{L^2} \quad \text{for } v\in L^2(\T).  
    \end{AlignedEquation}
    These estimates follow from the explicit representation of $e^{D\Delta t}$ as the kernel is given by $\mathbb Z$-periodization of the heat kernel satisfying the contractivity estimates \cite{rauch2012partial}. Consequently, for $0\leqslant t\leqslant \tau$, we derive the estimate 
    \begin{AlignedEquation}
        \|F(v)\|_{W^{1,2}} &\leqslant \| e^{D\Delta t}u_0 \|_{W^{1,2}} + C\int_0^t \left\| e^{D\Delta(t-s)}\left( \kappa\frac{e^{v(s,\cdot)}}{\intT e^{v(s,y)}dy}-v(s,\cdot)\right) \right\|_{W^{1,2}} ds \\
        &\leqslant \| e^{D\Delta t}u_0 \|_{W^{1,2}} + C\int_0^t \left(1+(t-s)^{-\frac{1}{2}}\right)  \left(\kappa \frac{\|e^{v(s,\cdot)}\|_{L^2}}{\|e^{v(s,\cdot)}\|_{L^1}} + \|v(s,\cdot)\|_{L^2}\right) ds.
    \end{AlignedEquation}
    Since $v$ is uniformly bounded, we have $\|e^v\|_{L^2} \leqslant e^{\|v\|_{L^\infty}}$ and $\|e^v\|_{L^1} \geqslant e^{-\|v\|_{L^\infty}}$. Thus, we obtain
    \begin{AlignedEquation}
        \|F(v)\|_{W^{1,2}}&\leqslant \| e^{D\Delta t}u_0 \|_{W^{1,2}} + C\int_0^t \left(1+(t-s)^{-\frac{1}{2}}\right)  \left(\kappa\frac{e^{\|v(s,\cdot)\|_{L^\infty}}}{e^{-\|v(s,\cdot)\|_{L^\infty}}} + \|v(s,\cdot)\|_2\right) ds \\
        &\leqslant \| e^{D\Delta t}u_0 \|_{W^{1,2}} + C(\kappa e^{2C\delta}+3\delta)\left(t+\sqrt{t}\right)
    \end{AlignedEquation}
    using $v\in V$ and that the semigroup $e^{D\Delta t}$ is bounded on $W^{1,2}(\T)$.
    By choosing $\tau>0$ sufficiently small, we obtain the estimate $\|F(v)\|_X \leqslant 3\delta$.
    
    Next we show that $F$ is a contraction mapping on the ball $V$. We have  
    \begin{AlignedEquation}
        \|F(v) &- F(w)\|_{W^{1,2}} \leqslant \int_0^t \left(1+(t-s)^{-\frac{1}{2}}\right) \left\| \kappa\frac{e^v}{\intT e^v dy} - \kappa\frac{e^w}{\intT e^w dy} - v + w \right\|_{L^2} ds.
    \end{AlignedEquation}
    Observing that ${e^v}/{\int e^v} - {e^w}/{\int e^w} = (e^v - e^w)/\int e^v + e^w(\int e^w - \int e^v)/(\int e^v \int e^w)$ and letting $C(t) = t+ \sqrt{t}$, we obtain 
    \begin{AlignedEquation}
        \|F(v) - F(w)\|_{W^{1,2}}&\leqslant C(t) \|v - w\|_X + C(t)\kappa\sup_{t\in[0,\tau]}\left\|\kappa\frac{e^v - e^w}{\intT e^v dy} + \kappa e^w\frac{\intT e^w dy - \intT e^v dy}{\intT e^v dy \intT e^w dy} \right\|_{L^2}.
    \end{AlignedEquation}
    Using $L^\infty$-estimates for $v$ and $w$ we derive \begin{AlignedEquation}
        e^w\left|\intT (e^w - e^v) dy\right|\big/\left(\intT e^v dy \intT e^w dy\right) \leqslant e^{3C\delta} \intT |e^v - e^w| dy
    \end{AlignedEquation}
    and 
    \begin{AlignedEquation}
        |e^v - e^w|/\intT e^v dy \leqslant e^{C\delta}|e^v - e^w| \; .
    \end{AlignedEquation}
    Thus, we obtain 
    \begin{AlignedEquation}
        \|F(v) - F(w)\|_{W^{1,2}} & \leqslant C(t) \|v - w\|_X + C(t)\kappa e^{C\delta}\sup_{t\in[0,\tau]}\left\|{e^v - e^w}\right\|_{L^2} + C(t)\kappa e^{3C\delta}\sup_{t\in[0,\tau]}\left\|{ e^v -  e^w} \right\|_{L^\infty}.
    \end{AlignedEquation}
    By the Sobolev embedding and mean value argument, we conclude
    \begin{AlignedEquation}
        \|F(v) - F(w)\|_{W^{1,2}}\leqslant C(t) \|v - w\|_X + C(t)\sup_{t\in[0,\tau]}\left\|{v - w}\right\|_{W^{1,2}} \leqslant C(\tau) \|v - w\|_X.
    \end{AlignedEquation}
    By choosing $\tau$ sufficiently small, the proof is completed by the Banach fixed point argument as $C(\tau) \downarrow 0$ as $\tau\downarrow0$. 

    Next, we demonstrate that the solution can be extended globally in time. The Lyapunov function defined in \eqref{eq:FunDef} satisfies
    \begin{AlignedEquation}
        \frac{d}{dt} \mathcal{J}\big(u(\cdot, t)\big) = D\mathcal{J}\big(u(\cdot, t)\big) u_t(\cdot,t) = -\|u_t(\cdot,t)\|_{L^2}^2\leqslant 0 \; , 
    \end{AlignedEquation}
    which implies that $\mathcal{J}\big(u(\cdot, t)\big) \leqslant \mathcal{J}\big(u_0\big)$. Using Sobolev embedding, we obtain
    \begin{AlignedEquation}
        \frac{D}{2}\intT\big|u_x(x,t)\big|^2 dx + \frac{1}{2} \intT \big|u(x,t)\big|^2 dx &\leqslant \mathcal{J}\big(u(\cdot,t)\big) + \kappa \log\left(\intT e^{u(x,t)} dx\right)
        \\&\leqslant \mathcal{J}\big(u_0\big) + \kappa C \|u(\cdot,t)\|_{W^{1,2}} \; .
    \end{AlignedEquation}
    This yields the inequality $\|u(\cdot,t)\|^2_{W^{1,2}}\leqslant C + \|u(\cdot,t)\|_{W^{1,2}}$ which implies $\|u(\cdot,t)\|_{W^{1,2}} <C\big(u_0\big)$ for all $t\geqslant \tau$. Similarly, we derive
    \begin{AlignedEquation}
        \int_0^t \|u_t(\cdot,t)\|_{L^2}^2  &= J(u_0) - J\big(u(\cdot,t)\big) \leqslant J(u_0) + C\|u(\cdot,t)\|_{W{1,2}} + \kappa C \|u(\cdot,t)\|_{L^\infty} \leqslant C \; ,
    \end{AlignedEquation}
    where the bound is uniform with respect to $t$. Furthermore, we have
    \begin{AlignedEquation}
        \int_0^t \|\Delta u(\cdot,t)\|_{L^2}^2 &= \int_0^t \left\|u_t(\cdot,t) + u(\cdot,t) + \kappa \frac{e^{u(\cdot,t)}}{\intT e^{u(y,t)} \dy}  \right\|_{L^2}^2  \\
        &\leqslant \int_0^t \left\|u_t(\cdot,t)\right\|_{L^2}^2 + \int_0^t \left\|u(\cdot,t)\right\|_{L^2}^2 + \int_0^t \left\|\kappa \frac{e^{u(\cdot,t)}}{\intT e^{u(y,t)} \dy} \right\|_{L^2}^2 \leqslant C + Ct.
    \end{AlignedEquation}
    These estimates ensures that the solution exists for all time $t\geqslant 0$ and satisfies
    \begin{AlignedEquation}
        u \in C\big([0,\infty), W^{1,2}(\T) \big), \qquad u\in L^2\big([0,\infty), L^2(\T)\big), \qquad u \in L^2\big((0,T), W^{2,2}(\T) \big),
    \end{AlignedEquation}
    for each $T>0$. To establish uniqueness, we employ the theory of dynamical systems. Since the solutions are bounded, it follows from \cite[Thm 3.4.1]{henry2006geometric} that the dependence of solutions of initial conditions is continuous, thereby defining dynamical system. Consequently the solutions are unique. 
\end{proof}

%\begin{proposition}
%    \label{prop:Nonnegativity}
 %   The solution $u(x,t)$ is nonnegative for nonnegative initial data $u_0(x)$.
%\end{proposition}

%\begin{proof}
%    Since $u$ satisfies
 %   \[
  %  u_t - D u_{xx} + u = \kappa \frac{e^u}{\int_0^1 e^{u(y,t)} dy} \geqslant 0,
   % \]
    %for nonegative initial conditions, the comparison principle for parabolic problems (see e.g., \cite{evans2022partial}) implies $u(x,t) \geqslant 0$ for all $(x,t) \in \mathbb{T} \times [0,\infty)$.
%\end{proof}

%\begin{proposition}
%    \label{thm:SolSym}
%    The solution $U(x)$ to equation \eqref{equ_varphi_Stat} is symmetric, namely $U(x) = U(1 - x)$.
%\end{proposition}

%\begin{proof}
 %   Let $\sigma = {\kappa}/{\intT e^{U}}$. Then $U(x)$ satisfies the second-order energy conservation equation with the potential $H({U}) = -{U^2}/{2} + \sigma e^{{U}}$. The orbits
 %   \begin{AlignedEquation}
        %\frac{U_x^2}{2} -\frac{U^2}{2} + \sigma e^{{U}} = E 
   % \end{AlignedEquation}
 %   for certain energy level are symmetric and hence by periodicity of boundary conditions, solutions $U(x)$ are also symmetric. 
%\end{proof}

\section{Computations of the derivatives}
\label{sec:deriv_G}
Define the mapping $G:X\times \R \to Y$ by
\begin{align*}
    G(\chi, \alpha) := D\chi_{xx} - \chi + (\kappa_n + \alpha) \left( \frac{e^\chi}{\int e^\chi d y} - 1 \right),
\end{align*}
where we use the notation of Section \ref{sec:Bifurcation_analysis}.
The derivatives of $G$ with respect to $\chi \in X$ and $\alpha \in \mathbb{R}$ are:
\begin{AlignedEquation}
    G_\chi(\chi, \alpha)(h) &= D h_{xx} - h + (\kappa_n + \alpha) \left[ \frac{e^\chi}{\int e^\chi d y} h - \frac{e^\chi}{(\int e^\chi dy )^2} \int e^\chi h dy \right],\\
    G_{\chi\chi}(\chi, \alpha)(h_1, h_2) &= (\kappa_n + \alpha) \left[ \frac{e^\chi}{\int e^\chi d y} h_1 h_2 + 2 \frac{e^\chi}{(\int e^\chi dy)^3} \int e^\chi h_1 dy \int e^\chi h_2 dy\right.\\
    &\quad \left.- \frac{e^\chi}{(\int e^\chi dy)^2} \left(h_1 \int e^\chi h_2 d y + h_2 \int e^\chi h_1 dy + \int e^\chi h_1 h_2 d y\right) \right],\\
    G_{\chi\chi\chi}(\chi, \alpha)(h,h,h) &= (\kappa_n + \alpha) \left[ \frac{e^\chi}{\int e^\chi dy} h^3 - \frac{e^\chi}{(\int e^\chi dy)^2} \left(3 h^2 \int e^\chi h dy + 3 h\int e^\chi h^2 dy + \int e^\chi h^3 d y\right)\right.\\
    &\quad+ \frac{e^\chi}{(\int e^\chi dy)^3} \left(  8 \int e^\chi h dy \int e^\chi h^2 dy + 6 h \left(\int e^\chi h dy\right)^2 + 2 \int e^\chi h dy \int e^\chi h^2 d y \right)\\
    &\quad\left. - 6 \frac{e^\chi}{\int e^\chi dy} (\int e^\chi h dy)^3 \right],\\
    G_\alpha(\chi, \alpha) &= \left(\frac{e^\chi}{\int e^\chi dy} - 1\right) ,\\
    G_{\chi\alpha}(\chi, \alpha) (h) &= \frac{e^\chi}{\int e^\chi dy} h - \frac{e^\chi}{(\int e^\chi dy)^2} \int e^\chi h dy
\end{AlignedEquation}
for $h, h_1, h_2 \in X$ and $\beta \in \R$.

\end{document}